\documentclass[a4paper,12pt]{article}

\usepackage[intlimits]{amsmath}
\usepackage{amssymb}
\usepackage{amsthm}
\usepackage{mathrsfs}
\usepackage{cite}
\usepackage{color}
\usepackage{bm}

\newcommand{\la}{\langle}
\newcommand{\ra}{\rangle}
\newcommand{\pr}{\partial}
\newcommand{\dom}{\Omega}

\newcommand{\im}{\mathfrak{Im}\;}

\newcommand{\vc}{\boldsymbol}

\newcommand{\R}{\mathbb{R}}
\newcommand{\C}{\mathbb{C}}

\newcommand{\dd}{\,\text{d}}

\newtheorem{thm}{Theorem}
\newtheorem{prop}{Proposition}[section]
\newtheorem{lem}{Lemma}[section]

\title{An inverse boundary value problem for certain anisotropic quasilinear elliptic equations}

\author{C\u{a}t\u{a}lin I. C\^{a}rstea\thanks{School of Mathematics, Sichuan University, Chengdu, Sichuan, 610064, P.R.China; email: catalin.carstea@gmail.com} \and Ali Feizmohammadi\thanks{Department of Mathematics, University College London, London, UK-WC1E 6BT, United Kingdom; email: a.feizmohammadi@ucl.ac.uk}}\date{}

\begin{document}
\maketitle

\begin{abstract}
In this paper we prove uniqueness in the inverse boundary value problem for quasilinear elliptic equations whose linear part is the Laplacian and nonlinear part is the divergence of a function analytic in the gradient of the solution. The main novelty in terms of the result is that the coefficients of the nonlinearity are allowed to be ``anisotropic''. As in previous works, the proof reduces to an integral identity involving the tensor product of the gradients of 3 or more harmonic functions. Employing a construction method using Gaussian quasi-modes, we obtain a convenient family of harmonic functions to plug into the integral identity and establish our result.
\end{abstract}

\section{Introduction}

Let $\dom\subset\R^{1+n}$, $n\geq2$, be an open, bounded domain, with $C^{1,1}$ boundary. (The slightly unusual choice to work in dimension $1+n$ rather than $n$ is due to the origin in hyperbolic equations of some of the methods we will employ.)

In $\dom$ we wish to consider quasilinear boundary value problems of the form
\begin{equation}\label{J-eq}
\left\{\begin{array}{l}\nabla\cdot\vc{\mathcal{J}}(\vc x, u,\nabla u)=0,\\[5pt] u|_{\pr\dom}=f.\end{array}\right.
\end{equation}
Here $\vc{\mathcal{J}}:\dom\times\R\times\R^{1+n}\to\R^n$ is  a function of suitable regularity. 

As an illustration, we may interpret \eqref{J-eq} to be a model for electrical conduction, with $u$ being the electric potential and $\vc{\mathcal{J}}$ being the current density.  A linear medium is one for which Ohm's law $\vc{\mathcal{J}}(\vc x,s,\vc p)=\sigma(\vc x)\vc p$ holds, where $\sigma$ is a (possibly matrix valued) function with a positive lower bound. A more general form for the function $\vc{\mathcal{J}}(\vc x,s,\vc p)$ would then correspond to nonlinear conductive media. 

To the problem \eqref{J-eq}, assuming we can show solutions exist and have sufficient regularity, we can associate the Dirichlet-to-Neumann map
\begin{equation}
\Lambda_{\vc{\mathcal{J}}}f=\left.\vc\nu\cdot\vc{\mathcal{J}}(\vc x,u,\nabla u)\right|_{\pr\dom},
\end{equation}
where $\vc\nu$ is the outer-pointing unit normal to $\pr\dom$, and $u$ is the solution of \eqref{J-eq}. In the conduction phenomena interpretation of \eqref{J-eq}, the Dirichlet-to-Neumann map gives for each potential distribution on the boundary the corresponding current density through the boundary. It is therefore a natural mathematical object to represent the totality of possible experimental data that may in principle be gathered from boundary measurements.

The inverse boundary value problem proposed by Calder\'on in \cite{Ca} is to invert the correspondence ${\vc{\mathcal{J}}}\to\Lambda_{\vc{\mathcal{J}}}$. The important subproblem of uniqueness is the question of injectivity for the correspondence ${\vc{\mathcal{J}}}\to\Lambda_{\vc{\mathcal{J}}}$. In general uniqueness cannot hold, as was first noted by Luc Tartar (account given in \cite{KV}) for linear problems. Suppose $\Phi:\overline\dom\to\overline\dom$ is a smooth diffeomorphism such that $\Phi(\vc x)=\vc x$ for all $\vc x\in\pr\dom$. Let
\begin{equation}
[\Phi_*\vc{\mathcal{J}}](\vc x,s,\vc p)=\left[(\det D\Phi)^{-1}(D\Phi)^t\vc{\mathcal{J}}\left(\cdot,s,(D\Phi)\vc p\right)\right]\circ\Phi^{-1}(\vc x).
\end{equation}
Then (see \cite{SuU}, \cite{S3})
\begin{equation}
\Lambda_{\Phi_*\vc{\mathcal{J}}}=\Lambda_{\vc{\mathcal{J}}}.
\end{equation}
It is conjectured that this is the only obstruction to uniqueness. However, except in dimension $1+n=2$ (e.g. \cite{S3}) there are only very partial results even in the linear case.

In this paper we will avoid the possibility for non-uniquenes by considering functions $\vc{\mathcal{J}}$ of a less general kind. One restriction will be to take  $\vc{\mathcal{J}}(\vc x,s,\vc p)$ to be a polynomial  in $\vc p$ (or an analytic function in $\vc p$) with coefficients depending on $\vc x$, with a zero zero order coefficient. The other restriction will be that the first order term will be $\vc p$. This is equivalent to saying that the linear part of equation \eqref{J-eq} is the Laplacian $\Delta$. This restriction in particular makes the obstruction to uniqueness described above impossible, as any nontrivial diffeomorphism would transform the Laplacian term into a second order differential operator with non-constant coefficients. 

For elliptic semilinear or quasilinear equations, there are a number of uniqueness results that are known. For semilinear equations, examples include \cite{IN}, \cite{IS}, \cite{S2}, \cite{FO}, \cite{LLLS1}, \cite{LLLS2}, \cite{KU}. For quasilinear equations, not in divergence form, see \cite{I2}. For  quasilinear equations in divergence form, when $\vc{\mathcal{J}}(\vc x, s, \vc p)=  c(\vc x,s)\vc p$ see \cite{S1}, \cite{SuU}, \cite{EPS}; when $\vc{\mathcal{J}}(\vc x,s,\vc p)=A(\vc x,s)\vc p$ with $A$ a matrix, see \cite{S3}; when $\vc{\mathcal{J}}(\vc x, s, \vc p)= \vc A(s,\vc p)$, see \cite{MU}, \cite{Sh}; when $\vc{\mathcal{J}}(\vc x, s, \vc p)=\vc A(\vc x,\vc p)$ in 2D see \cite{HS}; when $\vc{\mathcal{J}}(\vc x, s, \vc p)=\sigma(\vc x)\vc p+\vc b(\vc x)|\vc p|^2$, see \cite{KN}, \cite{CNV}; when $\vc{\mathcal{J}}(\vc x, s, \vc p)=\sigma(\vc x)\vc p+a(\vc x)|\vc p|^{q-2}\vc p$, $q\in(1,2)\cup(2,\infty)$, see \cite{CK}.   We can also mention  \cite{C} for quasilinear time-harmonic Maxwell systems.

\subsection{Assumptions and notation}

 Let $R>0$ be such that $\dom\subset B_R^{\R^{1+n}}=\{\vc x\in\R^{1+n}:|\vc x|<R\}$. For the coordinates of $\vc x\in\dom$ we will use the notation
\begin{equation}
\vc x=( x_0, x_1, x_2,\ldots, x_n)=( x_0,\vc x')=( x_0, x_1,\vc x'')=( x_0, x_1, x_2,\vc x'''),
\end{equation}
and correspondingly introduce the differential operators
\begin{equation}
\nabla=(\pr_0,\pr_1,\ldots,\pr_n),\quad \nabla'=(0,\pr_1,\ldots,\pr_n)
\end{equation}
and
\begin{equation}
\Delta=\sum_{j=0}^n\pr_j^2,\quad \Delta'=\sum_{j=1}^n\pr_j^2.
\end{equation}
Operators $\nabla''$, $\nabla'''$, $\Delta''$, $\Delta'''$ can also be defined in the same way.

In this paper we will assume that $\vc{\mathcal{J}}$ has the form
\begin{equation}\label{J-cond-1}
\vc{\mathcal{J}}(\vc x,s,\vc p)=\vc p+\sum_{k=2}^N \vc{\mathcal{J}}_k(\vc x,\vc p)+\vc{\mathcal{R}}(\vc x,s,\vc p),
\end{equation}
where
\begin{equation}\label{J-cond-2}
\vc{\mathcal{J}}_k(\vc x,\vc p)= \sum_{i_1,\ldots,i_k=0}^n \vc J_{k;i_1\cdots i_k}(\vc x)p_{i_1}\cdots p_{i_k}=\vc J_k:{\vc p}^{\otimes k},
\end{equation}
\begin{equation}\label{J-cond-3}
\vc J_{k;i_1\cdots i_k}\in C_0^{\infty}(\dom;\R^{1+n}),\quad\max_{k=2,N}||\vc J_{k;i_1\cdots i_k}||_{W^{1,\infty}(\dom)}<L,
\end{equation}
and $\mathcal{R}$ satisfies
\begin{equation}\label{J-cond-4}
|D_{\vc x} \vc{\mathcal{R}}(\vc x,s,\vc p)|+|\pr_s \vc{\mathcal{R}}(\vc x,s,\vc p)|<L |\vc p|^{N+1},\quad |D_p \vc{\mathcal{R}}(\vc x,s,\vc p)|<L|\vc p|^N,
\end{equation}
\begin{equation}\label{J-cond-5}
|D_{\vc x}D_p \vc{\mathcal{R}}(\vc x,s,\vc p)|+|\pr_s D_p\vc{\mathcal{R}}(\vc x,s,\vc p)|<L |\vc p|^{N},\quad |D^2_p \vc{\mathcal{R}}(\vc x,s,\vc p)|<L|\vc p|^{N-1}.
\end{equation}
Here $L$ is a finite positive constant. By ``$:$'' we denote the contraction of two tensors (e.g. see equation \eqref{J-cond-2}). We would like to point out that the coefficients $\vc J_{k;i_1\cdots i_k}(\vc x)$ must be symmetric under any permutation of the indices $i_1,\ldots,i_k$.

The boundary value problem \eqref{J-eq} then becomes
\begin{equation}\label{eq}
\left\{\begin{array}{l}\Delta u+\sum_{k=2}^N\nabla\cdot\vc{\mathcal{J}}_k(\vc x,\nabla u)+\nabla\cdot\vc{\mathcal{R}}(\vc x,u,\nabla u)=0,\\[5pt] u|_{\pr\dom}=f.\end{array}\right.
\end{equation}
Before proceeding to our main result we need to say something about existence of solutions. The following result for the existence of strong solutions is not novel. We provide a proof for the sake of completeness and for the convenience of the reader.
\begin{prop}\label{forward-prop}
Let $p\in(n,\infty)$. There exist $\kappa, K>0$ depending on $N$ and $L$, such that if $||f||_{W^{2-\frac{1}{p},p}(\pr\dom)}<\kappa$, then \eqref{eq} has a unique solution $u\in W^{2,p}(\dom)$ which satisfies
\begin{equation}
||u||_{W^{2,p}(\dom)}\leq K||f||_{W^{2-\frac{1}{p},p}(\pr\dom)}.
\end{equation}
\end{prop}
\begin{proof}
This follows from a standard contraction principle argument. We sketch it out here for the convenience of the reader. In what follows $C$ will stand for a number  of different positive constants depending on $\dom$, $N$,  and $L$.

Let $G_0: L^p(\dom)\to W^{2,p}_0(\dom)$ be the solution operator to the equation $\Delta u=F$, with zero Dirichlet boundary conditions. This is a bounded linear operator. (This is where we need the boundary to be $C^{1,1}$.) Given $f\in W^{2-\frac{1}{p},p}(\pr\dom)$, let $u_f\in W^{2,p}(\dom)$ be the harmonic function that is equal to $f$ on the boundary $\pr\dom$. Consider then the mapping
\begin{equation}
{T}_f(v)=u_f-G_0\left( \sum_{k=2}^N\nabla\cdot\vc{\mathcal{J}}_k(\vc x,\nabla v)+\nabla\cdot\vc{\mathcal{R}}(\vc x,v,\nabla v)\right).
\end{equation}
If $v\in W^{2,p}(\dom)$, by Sobolev embedding
\begin{equation}
||\nabla\cdot\vc{\mathcal{J}}_k(\vc x,\nabla v)||_{L^p(\dom)}\leq C ||v||_{W^{2,p}(\dom)}^k,
\end{equation}
\begin{equation}
||\nabla\cdot\vc{\mathcal{R}}(\vc x,v,\nabla v)||_{L^p(\dom)}\leq C ||v||_{W^{2,p}(\dom)}^{N+1}.
\end{equation}
It follows that $T_f:W^{2,p}(\dom)\to W^{2,p}(\dom)$, with
\begin{equation}
||T_f(v)||_{W^{2,p}(\dom)}\leq C\left(||f||_{W^{2-\frac{1}{p},p}(\pr\dom)} + \max(||v||_{W^{2,p}(\dom)}^2,||v||_{W^{2,p}(\dom)}^{N+1}) \right).
\end{equation}
Assuming $||f||_{W^{2-\frac{1}{p},p}(\pr\dom)}<1/4(1+C)$, we see that $T_f$ maps the ball of radius $1/2(1+C)$ in $W^{2,p}(\dom)$ to itself.

It remains to show that $T_f$ is a contraction. To that end observe that for $v,w\in W^{2,p}(\dom)$, contained in the ball just mentioned,
\begin{multline}
||\nabla\cdot\left(\vc{\mathcal{J}}_k(\vc x,\nabla v)-\vc{\mathcal{J}}_k(\vc x,\nabla w)\right)||_{L^p(\dom)}\\[5pt]\leq C\max(||v||_{W^{2,p}(\dom)},||w||_{W^{2,p}(\dom)})^{k-1}||v-w||_{W^{2,p}(\dom)},
\end{multline}
\begin{multline}
||\nabla\cdot\left(\vc{\mathcal{R}}(\vc x,v,\nabla v)-\vc{\mathcal{R}}(\vc x,w,\nabla w)\right)||_{L^p(\dom)}\\[5pt]\leq C\max(||v||_{W^{2,p}(\dom)},||w||_{W^{2,p}(\dom)})^{N-1}||v-w||_{W^{2,p}(\dom)}.
\end{multline}
It is then possible to choose a $\kappa>0$ such that if $||f||_{W^{2-\frac{1}{p},p}(\pr\dom)}<\kappa$, then the ball of radius $2\kappa$ in $W^{2,p}(\dom)$ is both invariant under $T_f$, and $T_f$ is a contraction on this ball. The unique fixed point of $T_f$ is the solution we are searching for.
\end{proof}

We can now define the Dirichlet-to-Neumann map. For $||f||_{W^{2-\frac{1}{p},p}(\pr\dom)}<\kappa$
\begin{equation}
\Lambda(f)=\left[ \pr_{\vc\nu}u+ \sum_{k=2}^N\vc\nu\cdot\vc{\mathcal{J}}_k(\vc x,\nabla u)+\vc\nu\cdot\vc{\mathcal{R}}(\vc x,u,\nabla u)  \right]_{\pr\dom}\in W^{1-\frac{1}{p},p}(\pr\dom),
\end{equation}
where $u$ is the solution of \eqref{eq}.

\subsection{Results and outline}

The fact that the linear part of the equation \eqref{eq} is prescribed suggests that it is reasonable to expect uniqueness to hold in this case, as any diffeomorphism would change the linear term and therefore the transformed equation will be outside the class we are considering here.  
The main result of this paper is that this intuition is correct.
\begin{thm}\label{main-thm}
Suppose we have $\vc{\mathcal{J}}^{(1)}$, $\vc{\mathcal{J}}^{(2)}$ that satisfy \eqref{J-cond-1}-\eqref{J-cond-5} and additionally $\Lambda^{(1)}(f)=\Lambda^{(2)}(f)$ for all $f$ such that  $||f||_{W^{2-\frac{1}{p},p}(\pr\dom)}<\kappa$. Then it must hold that
\begin{equation}
\vc{\mathcal{J}}_k^{(1)}=\vc{\mathcal{J}}_k^{(2)}, \quad k=2,\ldots, N.
\end{equation}
\end{thm}

Up to a technical tool which we will describe below, we prove this theorem in section \ref{sec-main-thm}. It is by now common in works on inverse boundary value problems for semilinear and quasilinear elliptic equations to use a so called ``second linearization'' trick, originally employed in \cite{I1}. The main idea is to plug in  Dirichlet data of the form $\epsilon f$, with $\epsilon$ a small parameter, into the Dirichlet-to-Neumann map. One then derives an asymptotic expansion of the form 
\begin{equation}
\Lambda(\epsilon f)=\epsilon\Lambda_1 (f)+\epsilon^2\Lambda_2(f)+\epsilon^3\Lambda_3(f)+\cdots,
\end{equation}
where $\Lambda_1$ is the Dirichlet-to-Neumann map associated to the linear term in \eqref{eq}, and $\Lambda_k(f)$ is homogeneous of degree $k$ in $f$. It is clear that $\Lambda$ determines each of these $\Lambda_k$. We do this in subsection \ref{subsec-expansions}.

In subsection \ref{subsec-iii} we then proceed to iteratively determine (in the sense of uniqueness) the coefficients $\vc J_k$ from the $\Lambda_k$. We can convert the equality $\Lambda_k^{(1)}-\Lambda_k^{(2)}=0$ into an integral identity involving harmonic functions and the coefficients $\vc J_{k;i_1\cdots i_k}$.
Through a polarization trick we will reduce the problem to showing that if $C$ is a 3-tensor such that $C_{jkl}=C_{kjl}$ and 
\begin{equation}\label{basic-integral-identity}
\int_\dom C:\nabla w_1\otimes\nabla w_2\otimes\nabla w_3\dd \vc x=0
\end{equation}
for all harmonic functions $w_1$, $w_2$, $w_3$, then $C=0$.

The principal ingredient in the proof of Theorem \ref{main-thm} is then the following result, which is perhaps of independent interest on its own.
\begin{thm}\label{C-thm}
Suppose $C=(C_{jkl})$  is a 3-tensor such that $C\in L^1(\R^n)$, the support of $C$ is compact, $C_{jkl}=C_{kjl}$ for all $j,l,k=0,\ldots,n$ and
\begin{equation}\label{basic-integral-identity}
\int C:\nabla w_1\otimes\nabla w_2\otimes\nabla w_3\dd \vc x=0
\end{equation}
for all smooth functions $w_1$, $w_2$, $w_3$ which are harmonic on a fixed bounded neighborhood $\mathcal{O}$ of the support of $C$. Then $C=0$.
\end{thm}

We give a proof of this result in section \ref{sec-C-thm}. Here we would like to point out by a simple argument that it is sufficient to show the result for all $C\in C_0^\infty(\R^n)$ that satisfy the assumptions of the theorem. To see that this is the case, first observe that for $C$ as in Theorem \ref{C-thm} and for $\vc y\in\R^n$ small enough we have that
\begin{equation}
\int C(\vc y-\vc x):\nabla w_1(\vc x)\otimes\nabla w_2(\vc x)\otimes\nabla w_3(\vc x)\dd \vc x=0,
\end{equation}
for all smooth harmonic functions $w_1$, $w_2$, $w_3$ in a fixed neighborhood of the support of $C$. This is the case because a translation (or a reflection) of a harmonic function is still harmonic. Suppose now that $\varphi_\epsilon$ is a compactly supported smooth approximation of identity. It follows that
\begin{multline}
\int \left(\varphi_\epsilon*C\right)(\vc x):\nabla w_1(\vc x)\otimes\nabla w_2(\vc x)\otimes\nabla w_3(\vc x)\dd \vc x\\[5pt]=
\int \varphi_\epsilon(\vc y)\left(\int C(\vc y-\vc x):\nabla w_1(\vc x)\otimes\nabla w_2(\vc x)\otimes\nabla w_3(\vc x)\dd \vc x\right)\dd \vc y\\[5pt]=0,
\end{multline}
for all smooth harmonic functions $w_1$, $w_2$, $w_3$ in a fixed neighborhood of the support of $\varphi_\epsilon*C$. It is clear then, since $\varphi_\epsilon*C\to C$ in $L^1(\R^n)$ norm, that if we can prove the result for $C_0^\infty(\R^n)$ functions, then Theorem \ref{C-thm} follows.

It turns out that the often used complex geometric optics solutions introduced by Calder\'on in \cite{Ca} are not a sufficiently rich family of harmonic functions to use in \eqref{basic-integral-identity} in order to extract information on $C$. We instead use a construction based on Gaussian quasi-modes in hyperplanes perpendicular to an arbitrary chosen direction.

Gaussing quasi-modes are approximate asymptotic eigenfunctions for the Laplace-Beltrami operator $\Delta'$ over $\R^n$, that concentrate along straight lines. Due to the splitting $\Delta=\partial^2_{x_0}+\Delta'$, they allow us to construct harmonic functions over compact subsets of $\R^{1+n}$ that {\em concentrate} on planes. They  have been recently used to solve inverse problems for linear elliptic equations in \cite{KSa}, \cite{DKLS} and non-linear elliptic equations in \cite{FO}, \cite{LLLS1}.

We remark that the construction of Gaussian quasi-modes in elliptic equations is based on a classical analogue for hyperbolic equations, namely Gaussian beams. These are approximate solutions to the wave equation that concentrate on null geodesics. They were introduced in the works \cite{BU} \cite{Ralston} and have been used in the context of inverse problems in many works, see for example \cite{KKL} and the references therein. 
Solutions to equations of the form $\Delta u+qu=0$ that concentrate on planes have also been constructed, by a different method, in \cite{GU}, with applications to the inverse problem for Schr\"odinger operators with a potential. 

In subsection \ref{subsec-gauss-construction} we describe how this Gaussian quasi-modes construction of harmonic functions works in our situation. The family of harmonic functions we construct will depend on an assymptotic parameter $\lambda$ which will be made to go to infinity later in the argument. They also depend on a number of arbitrary parameters. There are assymptotic expansions (in $\lambda^{-1}$ and $ x_2$) for these harmonic functions and we compute a few of the first terms whose exact expressions we will need to use.

In subsection \ref{subsec-stationary} we plug the Gaussian quasi-mode harmonic functions we have constructed into \eqref{basic-integral-identity}. As we let the assymptotic parameter $\lambda\to\infty$ we use a stationary phase theorem to obtain an expansion in powers of $\lambda^{-1}$ of the left hand side. The coefficients of each power of $\lambda^{-1}$ must each be zero.

{ We will split the tensor $C$ into a part $S$ that is symmetric under all permutations of the  indices and a remainder $D$. In subsection \ref{subsec-S-zero} we use the first order in the assymptotic expansion given by the stationary phase theorem to show that the symmetric part $S$ must be zero. In subsection \ref{subsec-A-zero} we combine information obtained from the next two orders in the expansion to deduce that $C$ itself must also be zero, which then concludes the proof of the Theorem \ref{C-thm}. Our ability to independently vary the various parameters on which our harmonic functions depend will be key to extracting useful information from the various orders in the expansion.}

\section{Proof of Theorem \ref{main-thm}}\label{sec-main-thm}

In this section we will reduce the proof of Theorem \ref{main-thm} to Theorem \ref{C-thm}. We do this by a version of the ``second linearization'' argument, i.e. we take Dirichlet data of the form $\epsilon f$ and derive asymptotic expansions as $\epsilon\to0$ for the corresponding solutions of \eqref{eq}, and consequently for the quantity $\Lambda(\epsilon f)$. From the successive orders in $\epsilon$ of this last quantity, using Theorem \ref{C-thm}, we can derive the uniqueness result for all the $\vc{\mathcal{J}}_k$ terms. This kind of iterative approach was also used, for example, in \cite{C} for time-harmonic nonlinear Maxwell systems.

\subsection{Small data asymptotic expansions}\label{subsec-expansions}

The most commonly used method for addressing inverse boundary value problems for elliptic semilinear or quasilinear equations is to introduce boundary Dirichlet data of the form $\epsilon f$, then derive an asymptotic expansion as $\epsilon\to0$ for the Dirichlet-to-Neumann map. To this end, for an $f\in W^{2,p}(\dom)$, let $u_\epsilon\in W^{2,p}(\dom)$ be the solution to
\begin{equation}\label{eq-u-epsilon}
\left\{\begin{array}{l}\displaystyle\Delta u_\epsilon+\sum_{k=2}^N\nabla\cdot\vc{\mathcal{J}}_k(\vc x,\nabla u_\epsilon)+\nabla\cdot\vc{\mathcal{R}}(\vc x,u_\epsilon,\nabla u_\epsilon)=0,\\[5pt] u_\epsilon|_{\pr\dom}=\epsilon f.\end{array}\right.
\end{equation}
We make the the following Ansatz
\begin{equation}\label{u-epsilon-expansion}
u_\epsilon=\sum_{k=1}^N\epsilon^k u_k+g_\epsilon,
\end{equation}
where the functions $u_k\in W^{2,p}(\dom)$ are solutions of the boundary value problems
\begin{equation}
\left\{\begin{array}{l}\Delta u_1=0,\\[5pt] u_1|_{\pr\dom}= f,\end{array}\right.
\qquad
\left\{\begin{array}{l}\displaystyle\Delta u_k +\nabla\cdot\vc{\mathcal{J}}_k(\vc x,\nabla u_1)
+\nabla\cdot\vc{\mathcal{N}}_k=0,\\[5pt] u_k|_{\pr\dom}= 0,\quad k=2,\ldots,N, \end{array}\right.
\end{equation}
where
\begin{multline}
\vc{\mathcal{N}}_k=\sum_{l=2}^{k-1}\!\!\sum_{\substack{\,\\\alpha_1+\alpha_2+\cdots+\alpha_N=l\\\alpha_1+2\alpha_2+\cdots+N\alpha_N=k}}\!\!\!\!\!\frac{l!}{\alpha_1!\cdots\alpha_N!}\vc J_l:(\nabla u_1)^{\otimes\alpha_1}\otimes\cdots\otimes(\nabla u_N)^{\otimes\alpha_N} \\[5pt]
\!\!\!\!\!=\sum_{l=2}^{k-1}\!\!\!\sum_{\substack{\,\\\alpha_1+\alpha_2+\cdots+\alpha_{k-1}=l\\\alpha_1+2\alpha_2+\cdots+(k-1)\alpha_{k-1}=k}}\!\!\!\!\!\frac{l!}{\alpha_1!\cdots\alpha_{k-1}!}\vc J_l:(\nabla u_1)^{\otimes\alpha_1}\otimes\cdots\otimes(\nabla u_{k-1})^{\otimes\alpha_{k-1}} .
\end{multline}
These equations are derived by formally plugging in the expansion \eqref{u-epsilon-expansion} for $u_\epsilon$ into \eqref{eq-u-epsilon} and matching the terms with the same powers of $\epsilon$. 

With these definitions we get that $g_\epsilon$ satisfies
\begin{equation}
\left\{\begin{array}{l}\Delta g_\epsilon+\epsilon^{N+1}\nabla\cdot\vc{\mathcal{N}}_{N+1}+\nabla\cdot\vc{\mathcal{R}}(\vc x,u_\epsilon,\nabla u_\epsilon)=0,\\[5pt] g_\epsilon|_{\pr\dom}= 0,\end{array}\right.
\end{equation}
where $\vc{\mathcal{N}}_{N+1}$ is a polynomial in $\nabla u_1, \ldots, \nabla u_N$ with coefficients constructed from the $\vc J_k(\vc x)$ functions, with the property that there exists a $C_f>0$, independent of $\epsilon$, such that $||\nabla\cdot\vc{\mathcal{N}}_{N+1}||_{L^p(\dom)}<C_f$. By \eqref{J-cond-4}, \eqref{J-cond-5}, and Proposition \ref{forward-prop} we also have that
\begin{equation}
||\nabla\cdot\vc{\mathcal{R}}(\vc x,u_\epsilon,\nabla u_\epsilon)||_{L^p(\dom)}<C_f\epsilon^{N+1}.
\end{equation}
It follows that
\begin{equation}
||g_\epsilon||_{W^{2,p}(\dom)}<C_f\epsilon^{N+1}.
\end{equation}

The expansion \eqref{u-epsilon-expansion} implies an expansion for $\Lambda(\epsilon f)$. We have
\begin{equation}
\Lambda(\epsilon f)=\sum_{k=1}^N\epsilon^k\left[\pr_{\vc\nu}u_k+\vc\nu\cdot\vc{\mathcal{J}}_k(\vc x,\nabla u_1)+\vc\nu\cdot\vc{\mathcal{N}}_k \right]_{\pr\dom}+\mathscr{O}(\epsilon^{N+1}).
\end{equation}
Suppose $w\in W^{2,p}(\dom)$ is a harmonic function. Integration by parts gives that $\la w|_{\pr\dom}, \pr_{\vc\nu}u_k|_{\pr\dom}\ra=0$ for $k=2,\ldots, N$. Then
\begin{equation}\label{lambda-expansion}
\la w,\Lambda (\epsilon f)\ra=\epsilon\int_\dom\nabla w\cdot\nabla u_1
+\sum_{k=2}^N\epsilon^k\int_\dom\nabla w\cdot\left( \vc{\mathcal{J}}_k(\vc x,\nabla u_1)+\vc{\mathcal{N}}_k \right)
+\mathscr{O}(\epsilon^{N+1}).
\end{equation}

\subsection{Integral identities \& induction}\label{subsec-iii}

Suppose that we are under the assumptions of Theorem \ref{main-thm}. By \eqref{lambda-expansion} we must have
\begin{equation}
\int_\dom\nabla w_1\cdot\left( \vc{\mathcal{J}}^{(1)}_k(\vc x,\nabla u_1)-\vc{\mathcal{J}}^{(2)}_k(\vc x,\nabla u_1)+\vc{\mathcal{N}}_k^{(1)}- \vc{\mathcal{N}}_k^{(2)}\right)=0,
\end{equation}
for $k=2,\ldots,N$, and for any harmonic functions $w_1, u_1\in W^{2,p}(\dom)$.

The $k=2$ case is
\begin{equation}
\int_\dom\nabla w_1\cdot(\vc J_2^{(1)}-\vc J_2^{(2)}):\nabla u_1\otimes\nabla u_1=0.
\end{equation}
Suppose $w_2$, $w_3$ are harmonic functions. We can insert $u_1=tw_2+sw_3$, with $s$, $t$ real parameters, into the above identity. Considering only the coefficient of $ts$ in the resulting relation and using the symmetry of the $\vc J_{2;ij}$ coefficients we obtain that
\begin{equation}
\int_\dom\nabla w_1\cdot(\vc J_2^{(1)}-\vc J_2^{(2)}):\nabla w_2\otimes\nabla w_3=0.
\end{equation}

From Theorem \ref{C-thm} it follows that $\vc J_2^{(1)}=\vc J_2^{(2)}$. For the purpose of induction, suppose that $\vc J_l^{(1)}=\vc J_l^{(2)}$, for all $l=1,\ldots,k-1$. We then also have that $u_l^{(1)}=u_l^{(2)}$, for all $l=1,\ldots,k-1$, and therefore that $\vc{\mathcal{N}}_k^{(1)}=\vc{\mathcal{N}}_k^{(2)}$. We then arrive at the identity
\begin{equation}
\int_\dom\nabla w_1\cdot(\vc J_k^{(1)}-\vc J_k^{(2)}):(\nabla u_1)^{\otimes k}=0,
\end{equation}
for all harmonic functions $w_1$, $u_1$. As above, we can use polarization in the $(\nabla u_1)^{\otimes k}$ term. Let $w_2$, \ldots, $w_{k+1}$ be harmonic functions in $\dom$, $t_2$, \ldots, $t_{k+1}$ be real parameters. If we take $u_1=t_2w_2+\cdots+t_{k+1}w_{k+1}$, then the coefficient of $t_2\cdots t_{k+1}$ gives
\begin{equation}
\int_\dom\nabla w_1\cdot(\vc J_k^{(1)}-\vc J_k^{(2)}):\nabla w_2\otimes\cdots\otimes\nabla w_{k+1}=0.
\end{equation}
We can choose $w_4$, \ldots, $w_{k+1}$ to be coordinate functions (i.e. one of $ x_0$, \ldots, $ x_n$). For any fixed such choice, let $C$ be a 3-tensor such that
\begin{equation}
C:\nabla w_1\otimes\nabla w_2\otimes\nabla w_3=\nabla w_{k+1}\cdot(\vc J_k^{(1)}-\vc J_k^{(2)}):\nabla w_1\otimes\cdots\otimes\nabla w_{k}.
\end{equation}
Then \eqref{basic-integral-identity} is satisfied, which gives $\vc J_k^{(1)}=\vc J_k^{(2)}$.

The conclusion of Theorem \ref{main-thm}  follows by induction. All that remains is to provide a proof of Theorem \ref{C-thm}.

\section{Proof of Theorem \ref{C-thm}}\label{sec-C-thm}

In this section we give a proof of Theorem \ref{C-thm}. Recall that it is sufficient to consider a tensor $C$ that is smooth and compactly supported. In order to fix notation, we assume that $C\in C_0^\infty(\mathcal{O})$, where $\mathcal{O}$ is an opend and bounded subset of $\R^n$.

We begin with the construction of a family of harmonic functions, pa\-ra\-me\-trized by a complex parameter $\tau=\lambda+i\sigma$ (and a choice of coordinate axes). We use Gaussian quasi-modes in the variables $\vc x'$ in order to construct this family. We are here not only interested in the existence of such solutions, but we also compute for them expansions in $ x_2$ and $\tau$, for small $ x_2$ and large $\lambda$. 

We plug such solutions into the identity \eqref{basic-integral-identity} and in the limit $\lambda\to\infty$ we use the stationary phase theorem in order to obtain an expansion of the integral into powers of $\lambda^{-1}$. As each order in the expansion must be independently zero, using the first three terms we are able to extract enough information to show that $C=0$.

It will turn out to be convenient to split $C$ into fully symmetric and non-symmetric parts, with respect to the first two indices:
\begin{equation}
C=S+D, \quad S_{jkl}=\frac{1}{6}\sum_{\sigma\in\Sigma_3}C_{\sigma(jkl)}=\frac{1}{3}(C_{jkl}+C_{klj}+C_{ljk}).
\end{equation}
Note that $S$ is symmetric under all permutations of the indices.

\subsection{Gaussian quasi-modes construction of harmonic functions}\label{subsec-gauss-construction}

Note that without any loss of generality we may assume that $\mathcal{O}\subset\{ x_1>0\}$.

Let $\tau=\lambda+i\sigma$, $\lambda, \sigma\in\R$. We would like to find harmonic functions of the form
\begin{equation}\label{quasi-construction}
u_\tau^\pm(\vc x)=e^{\pm\lambda  x_0}\left( e^{\pm i\sigma  x_0} e^{i\tau\Psi( x_1, x_2)}a_\tau( x_1, x_2)+r_\tau(\vc x)   \right).
\end{equation}
Note that
\begin{equation}
\Delta u_\tau^\pm=e^{\pm\tau  x_0}[\tau^2+\Delta'] \left(e^{i\tau\Psi}a_\tau\right)
+\Delta \left(e^{\pm\lambda  x_0}r_\tau\right),
\end{equation}
and
\begin{multline}\label{quasi-op}
(\tau^2+\Delta') e^{i\tau\Psi}a_\tau\\[5pt]=
e^{i\tau\Psi}\left[ \tau^2(1-|\nabla'\Psi|^2)a_\tau+
i\tau(2\nabla'\Psi\cdot\nabla' a_\tau+(\Delta'\psi)a_\tau)
+(\Delta' a_\tau) \right].
\end{multline}

The quantity $\lambda$ will be an asymptotic parameter, in the sense that we eventually intend to take the limit $\lambda\to\infty$. We will construct $\Psi$ and $a_\tau$ so that this quantity vanishes to high order in both $\lambda^{-1}$ and $ x_2$.

\subsubsection{Expansions for $u_\tau$}

Let $M$ be a large natural number, $\delta>0$, let $\chi:\R\to[0,\infty)$ be a smooth function such that $\chi(t)=1$ for $|t|<\frac{1}{2}$ and $\chi(t)=0$ for $|t|>1$, and let $h:\R^{n-1}\to\R$ be a harmonic function.
We make the following Ans\"atze:
\begin{equation}\label{An-1}
\Psi( x_1, x_2)=\sum_{j=0}^M\psi_j( x_1) x_2^j,
\end{equation}
\begin{equation}\label{An-2}
a_\tau( x_1, x_2)=
\chi(\frac{ x_2}{\delta})h(\vc x''')
\sum_{k=0}^Mv_k( x_1, x_2)\tau^{-k},
\end{equation}
\begin{equation}\label{An-3}
v_k( x_1, x_2)=\sum_{j=0}^M v_{k;j}( x_1) x_2^j.
\end{equation}

Since we would like for the quantity in equation \eqref{quasi-op} to vanish to high order in $\lambda^{-1}$ and $ x_2$, we require that the  following conditions hold:
\begin{itemize}
\item[(i)]$\im \Psi\geq \kappa | x_2|^2$ ;
\item[(ii)]$\pr_2^j(|\nabla'\Psi|^2-1)|_{ x_2=0}=0$ for $j=0,1,\ldots,M$;
\item[(iii)]$\pr_2^j(2\nabla'\Psi\cdot\nabla' v_0+(\Delta'\Psi)v_0)|_{ x_2=0} =0$ for $j=0,1,\ldots,M$.
\item[(iv)]$\pr_2^j(2\nabla'\Psi\cdot\nabla' v_k+(\Delta'\Psi)v_k-i\Delta'v_{k-1})|_{ x_2=0} =0$ for $j=0,1,\ldots,M$.
\end{itemize}
As we will see, this conditions do not uniquely determine the $\psi_j$, $v_{k;j}$. They will however provide ODEs that these quantities need to satisfy. The general forms that we obtain for the first few of these quantities are collected in Appendix \ref{appendix-results}.

With the definitions above, we have
\begin{equation}
\pr_1\Psi=\sum_{j=0}^M\dot\psi_j  x_2^j,\quad \pr_2\Psi=\sum_{j=0}^{M-1}(j+1)\psi_{j+1} x_2^j,
\end{equation}
\begin{equation}
\pr_1v_k=\sum_{j=0}^M\dot v_{k;j}  x_2^j,\quad \pr_2 v_k=\sum_{j=0}^{M-1}(j+1)v_{k;j+1} x_2^j,
\end{equation}
\begin{equation}
\Delta'\Psi=\sum_{j=0}^{M-2}[\ddot\psi_j+(j+1)(j+2)\psi_{j+2}] x_2^j
+\sum_{j=2M-1}^{2M}\ddot\psi_j  x_2^j.
\end{equation}
Since
\begin{multline}
|\nabla'\Psi|^2=
\sum_{j=0}^{2M-2} x_2^j\sum_{l=0}^j\left[\dot\psi_l\dot\psi_{j-l}+(l+1)(j-l+1)\psi_{l+1}\psi_{j-l+1}   \right]\\[5pt]
+\sum_{j=2M-1}^{2M} x_2^j\sum_{l=0}^j\dot\psi_l\dot\psi_{j-l},
\end{multline}
requirement (ii) for $j=0$ is
\begin{equation}
\dot\psi_0^2+\psi_1^2=1.
\end{equation}
We make the choices
\begin{equation}
\psi_0( x_1)= x_1,\quad \psi_1( x_1)=0.
\end{equation}

For $j=1$ we have
\begin{equation}
2\dot\psi_0\dot\psi_1+4\psi_1\psi_2=0,
\end{equation}
which is satisfied automatically. 

For $j=2$ we have
\begin{equation}
2\dot\psi_0\dot\psi_2+\dot\psi_1^2+6\psi_1\psi_3+4\psi_2^2=0,
\end{equation}
which simplifies to
\begin{equation}
\dot\psi_2+2\psi_2^2=0.
\end{equation}
We choose
\begin{equation}
\psi_2( x_1)=\frac{1}{2}( x_1-i\epsilon)^{-1},
\end{equation}
which satisfies the equation for $\psi_2$ and also requirement (i). 

For $j=3$ we have
\begin{equation}
2\dot\psi_0\dot\psi_3+2\dot\psi_1\dot\psi_2+8\psi_1\psi_4+12\psi_2\psi_3=0,
\end{equation}
which simplifies to 
\begin{equation}
\dot\psi_3+3( x_1-i\epsilon)^{-1}\psi_3=0.
\end{equation}
Any multiple of $( x_1-i\epsilon)^{-3}$ is a solution
\begin{equation}
\psi_3=p_3( x_1-i\epsilon)^{-3},\quad p_3\in\C.
\end{equation}

For $j=4$ we have
\begin{equation}
2\dot\psi_0\dot\psi_4+2\dot\psi_1\dot\psi_3+\dot\psi_2^2+10\psi_1\psi_5+16\psi_2\psi_4+9\psi_3^2=0,
\end{equation}
which simplifies to 
\begin{equation}
\dot\psi_4+\frac{1}{8}( x_1-i\epsilon)^{-4}+4( x_1-i\epsilon)^{-1}\psi_4+\frac{9}{2}p_3^2( x_1-i\epsilon)^{-6}=0.
\end{equation}
The general solution to this equation is
\begin{equation}
\psi_4( x_1)=-\frac{1}{8}( x_1-i\epsilon)^{-3}+\frac{9}{2}p_3^2( x_1-i\epsilon)^{-5}+p_4( x_1-i\epsilon)^{-4},\quad p_4\in\C.
\end{equation}

For $j=5$ we have
\begin{equation}
2\dot\psi_5\dot\psi_0+2\dot\psi_4\dot\psi_1+2\dot\psi_3\dot\psi_2+12\psi_6\psi_1+20\psi_5\psi_2+24\psi_4\psi_3=0,
\end{equation}
which simplifies to
\begin{equation}
\dot\psi_5+ 5( x_1-i\epsilon)^{-1}\psi_5+54p_3^3( x_1-i\epsilon)^{-8}
+12p_3p_4( x_1-i\epsilon)^{-7}=0.
\end{equation}
The general solution to this equation is
\begin{multline}
\psi_5( x_1)
=27p_3^3( x_1-i\epsilon)^{-7}+12p_3p_4( x_1-i\epsilon)^{-6}
\\[5pt]
 +p_5( x_1-i\epsilon)^{-5},\quad p_5\in\C.
\end{multline}
We can continue, in principle, to solve like this to find all the functions $\psi_j$, $j=0,\ldots,M$.

In order to use requirement (iii), note that
\begin{multline}
2\nabla'\Psi\cdot\nabla' v_0+(\Delta'\Psi)v_0\\[5pt]=
\sum_{j=0}^{2M-2} x_2^j\sum_{l=0}^j\left[2\dot v_{0;l}\dot\psi_{j-l}+2(l+1)(j-l+1)v_{0;l+1}\psi_{j-l+1}
\right.\\[5pt]\left.+v_{0;l}\left(\ddot\psi_{j-l}+(j-l+1)(j-l+2)\psi_{j-l+2}\right)\right]\\[5pt]
+\sum_{j={2M-1}}^{2M} x_2^j\sum_{l=0}^j\left[2\dot v_{0;l}\dot\psi_{j-l}+ v_{0;l}\ddot\psi_{j-l} \right].
\end{multline}
For $j=0$, requirement (iii)  gives
\begin{equation}
2\dot v_{0;0}\dot\psi_0+2v_{0;1}\psi_1+v_{0;0}(\ddot\psi_0+2\psi_2)=0,
\end{equation}
which simplifies to
\begin{equation}
\dot v_{0;0}+\frac{1}{2}( x_1-i\epsilon)^{-1} v_{0;0}=0.
\end{equation}
We then choose 
\begin{equation}
v_{0;0}=( x_1-i\epsilon)^{-\frac{1}{2}}.
\end{equation}

For $j=1$ we have
\begin{equation}
2\dot v_{0;0}\dot\psi_1+2\dot v_{0;1}\dot\psi_0+4v_{0;1}\psi_2+4v_{0;2}\psi_1
+v_{0;0}(\ddot\psi_1+6\psi_3)+v_{0;1}(\ddot\psi_0+2\psi_2)=0,
\end{equation}
which simplifies to
\begin{equation}
\dot v_{0;1}+\frac{3}{2}( x_1-i\epsilon)^{-1}v_{0;1}+3p_3( x_1-i\epsilon)^{-\frac{7}{2}}=0.
\end{equation}
The general solution to this equation is
\begin{equation}
v_{0;1}=3p_3( x_1-i\epsilon)^{-\frac{5}{2}}+q_1( x_1-i\epsilon)^{-\frac{3}{2}},\quad q_1\in\C.
\end{equation}

All other $v_{k;j}$ can also be determined in the same manner.

\subsubsection{Estimate for the remainder term}

If we want the definitions \eqref{quasi-construction}, \eqref{An-1}-\eqref{An-3}, to be meaningful, we need to show that the remainder term  $r_\tau$ can be controlled as $\lambda\to\infty$.

In order to see by how much $e^{\tau  x_0}e^{i\tau\Psi}a_{\tau}$ fails to be harmonic, first note the following basic calculus result:
\begin{equation}\label{calc}
\tau^m e^{-\tau \kappa  x_2^2}\leq C_m  x_2^{-2m}.
\end{equation}
By requirements (i)-(iv), and also by making sure that $\inf\{ x_1:\vc x\in\mathcal{O}\}>0$, we  have
\begin{equation}
\left|e^{i\tau\Psi}\tau^2(1-|\nabla'\Psi|^2)a_\tau\right|\leq C|\tau|^{-\frac{M-3}{2}},
\end{equation}
and in fact more generally (for fixed $\delta$!)
\begin{equation}
\left|\pr_j^k\left[e^{i\tau\Psi}\tau^2(1-|\nabla'\Psi|^2)a_\tau\right]\right|\leq C|\tau|^{-\frac{M-3-3k}{2}}.
\end{equation}
Similarly
\begin{equation}
\left|\pr_{j}^ki\tau e^{i\tau\Psi} \left[(2\nabla'\Psi\cdot\nabla' a_\tau+(\Delta'\psi)a_\tau)
+(\Delta' a_\tau)\right]\right|\leq C |\tau|^{-\frac{M-1-3k}{2}}.
\end{equation}
In the above estimate we need to be careful with the terms involving derivatives of $\chi$ or $h$ that are generated by the $\nabla' a_\tau$ and $\Delta'a_\tau$ terms. First note that there is in fact no such term involving derivatives of $h$ since $\nabla'\Psi\cdot\nabla' a_\tau$ only contains derivatives with respect to $ x_1$ and $ x_2$, and also $\Delta''' a_\tau=0$ as $h$ is harmonic. As for the terms containing derivatives of $\chi$, those are all supported in $\{\frac{\delta}{2}\leq  x_2\leq \delta\}$, so they may be controlled by a bound of the form $C e^{-\tau\kappa\delta^2}$.

It follows that
\begin{equation}
||(\tau^2+\Delta') e^{i\tau\Psi}a_\tau||_{H^k(\mathcal{O})}=\mathscr{O}(|\tau|^{-\frac{M-3-3k}{2}}).
\end{equation}

The remainder term $r_\tau$ satisfies the equation
\begin{equation}
e^{\mp\lambda  x_0}\Delta \left(e^{\pm\lambda  x_0} r_\tau\right)
=-e^{\pm i\sigma  x_0}[\tau^2+\Delta'] \left(e^{i\tau\Psi}a_\tau\right).
\end{equation}
We have 
\begin{equation}
||e^{\mp\lambda  x_0}\Delta \left(e^{\pm\lambda  x_0} r_\tau\right)||_{H^k(\mathcal{O})}
\leq C_{k,\sigma}||[\tau^2+\Delta'] \left(e^{i\tau\Psi}a_\tau\right)||_{H^k(\mathcal{O})}=\mathscr{O}(|\tau|^{-\frac{M-3-3k}{2}}).
\end{equation}
It then follows (e.g. \cite[Proposition 2]{FO}) that there exists an $r_\tau$ such that
\begin{equation}
||r_\tau||_{H^{k}(\mathcal{O})}=\mathscr{O}(|\tau|^{-\frac{M-1-3k}{2}}).
\end{equation}

\subsection{Application of a stationary phase theorem}\label{subsec-stationary}

Here we will plug  the harmonic funcitons constructed above into the identity \eqref{basic-integral-identity}. Before doing so, it is useful to start with a few preparatory calculations. 

We observe that
\begin{equation}
\nabla u_\tau^\pm=\tau e^{\pm\tau  x_0}e^{i\tau\Psi}(\pm a_\tau\vc e_0+i(\nabla'\Psi)a_\tau+\tau^{-1}\nabla' a_\tau+\nabla r_\tau).
\end{equation}
This can be expanded in powers of $\tau^{-1}$ and $ x_2$ as follows
\begin{multline}
V_\tau^+:=\tau^{-1} e^{-\tau  x_0}e^{-i\tau\Psi}\nabla u_\tau^+\\[5pt]
=\chi h\Big[(v_{0;0}+v_{0;1} x_2+v_{0;2} x_2^2+\tau^{-1}v_{1;0}+v_{0;3} x_2^3+v_{1;1}\tau^{-1} x_2)\vc e_0\\[5pt]
+\left[iv_{0;0}+iv_{0;1} x_2+i(v_{0;2}+v_{0;0}\dot\psi_2) x_2^2+(iv_{1;0}+\dot v_{0;0})\tau^{-1}\right.\\[5pt]\left.
+i(v_{0;3}+v_{0;1}\dot\psi_2+v_{0;0}\dot\psi_3) x_2^3+(iv_{1;1}+\dot v_{0;1})\tau^{-1} x_2   \right]\vc e_1\\[5pt]
+\left[2iv_{0;0}\psi_2 x_2+i(2v_{0;1}\psi_2+3v_{0;0}\psi_3) x_2^2+v_{0;1}\tau^{-1}\right.\\[5pt]\left.
+i(2v_{0;2}\psi_2+3v_{0;1}\psi_3+4v_{0;0}\psi_4) x_2^3 + (2iv_{1;0}\psi_2+2v_{0;2})\tau^{-1} x_2   \right]\vc e_2\Big]\\[5pt]
+\tau^{-1}\left[ \delta^{-1}\dot\chi h\vc e_2+\chi\nabla''' h\right](v_{0;0}+v_{0;1} x_2)+\cdots ,
\end{multline}
where the omitted terms contain a power of $ x_2$ larger than 3, or a power of $\tau^{-1}$ larger than 1, or $\tau^{-1}$ times a power of $ x_2$ larger than 1.
Let $\vc\alpha=\vc e_0+i\vc e_1$.
Regrouping terms we can write
\begin{multline}\label{V-plus}
V_\tau^+=\chi h\Big[v_{0;0}\vc\alpha+ x_2\left[v_{0;1}\vc\alpha+2iv_{0;0}\psi_2\vc e_2  \right]\\[5pt]
+ x_2^2\left[v_{0;2}\vc\alpha+iv_{0;0}\dot\psi_2\vc e_1+i(2v_{0;1}\psi_2+3v_{0;0}\psi_3)\vc e_2  \right]\\[5pt]
+\tau^{-1}\left[v_{1;0}\vc\alpha +\dot v_{0;0}\vc e_1+v_{0;1}\vc e_2 \right]\\[5pt]
+ x_2^3\left[v_{0;3}\vc\alpha+i(v_{0;1}\dot\psi_2+v_{0;0}\dot\psi_3)\vc e_1 +i(2v_{0;2}\psi_2+3v_{0;1}\psi_3+4v_{0;0}\psi_4)\vc e_2   \right]\\[5pt]
+\tau^{-1} x_2\left[ v_{1;1}\vc\alpha+\dot v_{0;1}\vc e_1 +(2iv_{1;0}\psi_2+2v_{0;2})\vc e_2 \right]\Big]\\[5pt]
+\tau^{-1}\left[ \delta^{-1}\dot\chi h\vc e_2+\chi\nabla''' h\right](v_{0;0}+v_{0;1} x_2)+\cdots,
\end{multline}
and similarly
\begin{multline}\label{V-minus}
(-1)\overline{V_\tau^-}:=(-1)\overline{\tau^{-1} e^{\tau  x_0}e^{i\tau\Psi}\nabla u_\tau^-}\\[5pt]
=\chi h\Big[\overline{v_{0;0}}\vc\alpha+ x_2\left[\overline{v_{0;1}}\vc\alpha+2i\overline{v_{0;0}}\overline{\psi_2}\vc e_2  \right]\\[5pt]
+ x_2^2\left[\overline{v_{0;2}}\vc\alpha+i\overline{v_{0;0}}\overline{\dot\psi_2}\vc e_1+i(2\overline{v_{0;1}}\overline{\psi_2}+3\overline{v_{0;0}}\overline{\psi_3})\vc e_2  \right]\\[5pt]
+\overline{\tau^{-1}}\left[\overline{v_{1;0}}\vc\alpha -\overline{\dot v_{0;0}}\vc e_1-\overline{v_{0;1}}\vc e_2 \right]\\[5pt]
+ x_2^3\left[\overline{v_{0;3}}\vc\alpha+i(\overline{v_{0;1}}\overline{\dot\psi_2}+\overline{v_{0;0}}\overline{\dot\psi_3})\vc e_1 +i(2\overline{v_{0;2}}\overline{\psi_2}+3\overline{v_{0;1}}\overline{\psi_3}+4\overline{v_{0;0}}\overline{\psi_4})\vc e_2   \right]\\[5pt]
+\overline{\tau^{-1}} x_2\left[ \overline{v_{1;1}}\vc\alpha-\overline{\dot v_{0;1}}\vc e_1 +(2i\overline{v_{1;0}}\overline{\psi_2}-2\overline{v_{0;2}})\vc e_2 \right]\Big]\\[5pt]
-\overline{\tau^{-1}}\left[ \delta^{-1}\dot\chi h\vc e_2+\chi\nabla''' h\right](\overline{v_{0;0}}+\overline{v_{0;1}} x_2)+\cdots.
\end{multline}

We choose
\begin{equation}
w_1=u_{1;\tau}^+,\quad w_2=u_{2;\tau}^+,\quad w_3=\overline{u_{3;2\tau}^-},
\end{equation}
where the numerical indices indicate that we might have different choices in the constants $p_j$,etc., and different choices of harmonic functions $h_1$, $h_2$, and $h_3$. For the various constants that appear in the expansions of our special solutions we will use a superscript to indicate the solution to which it belongs. For example $p_1^2$ will be the $p_1$ constant associated to solution 2.

For convenience, we choose $p_3^j=p_4^j=p_5^j=0$ for $j=1,2,3$. This will simplify somewhat the computations that will follow below.

With these choices
\begin{equation}
e^{\tau  x_0}e^{i\tau\Psi_1}e^{\tau  x_0}e^{i\tau\Psi_2}e^{-2\overline{\tau}  x_0}e^{-2i\overline{\tau}\overline{\Psi_3}}
=e^{4i\sigma  x_0}e^{i(\tau\Psi_1+\tau\Psi_2-2\overline{\tau}\overline{\Psi_3})},
\end{equation}
where we can expand
\begin{equation}
\tau\Psi_1+\tau\Psi_2-2\overline{\tau}\overline{\Psi_3}
=4i\sigma  x_1+\sigma F_1+\lambda F_2,
\end{equation}
with
\begin{equation}
F_1=\frac{2i x_1}{ x_1^2+\epsilon^2} x_2^2-\frac{i}{2}\frac{ x_1^3-3 x_1\epsilon^2}{( x_1^2+\epsilon^2)^3} x_2^4+\mathscr{O}( x_2^6),
\end{equation}
\begin{equation}
F_2=\frac{2i\epsilon}{ x_1^2+\epsilon^2} x_2^2+\frac{i}{2}\frac{\epsilon^3-3 x_1^2\epsilon}{( x_1^2+\epsilon^2)^3} x_2^4+\mathscr{O}( x_2^6).
\end{equation}

We can write
\begin{multline}
\tau^{-2}\overline{\tau}^{-1}\frac{1}{2}C : \nabla w_1\otimes\nabla w_2\otimes\nabla w_3
\\[5pt]
=e^{i\lambda F_2}e^{4i\sigma  x_0}e^{-4\sigma  x_1}e^{i\sigma F_1}C :  V_{1;\tau}^+\otimes V_{2;\tau}^+\otimes\overline{V_{3;2\tau}^-}.
\end{multline}
We would like to quote here a suitable stationary phase theorem, slightly addapted from the source to fit our particular circumstances.
\begin{thm}[{see \cite[Theorem 7.7.5]{H1}}]\label{H-thm}
Let $0\in I\subset \R$ be a bounded interval. Let $X$ be an open neighborhood of $\overline{I}$. If $U\in C_{0}^{2k}(I)$, $F\in C^{3k+1}(X)$ and $\im F\geq0$ in $X$, $F'(0)=0$, $ F''(0)\neq 0$, $F'\neq0$ in $I\setminus\{0\}$. Let
\begin{equation}
G(t)=F(t)-F(0)-\frac{1}{2}F''(0)t^2
\end{equation}
and
\begin{equation}
L_jU=\sum_{\nu-\mu=j}\sum_{2\nu\geq3\mu}i^{-j}\frac{1}{\nu!\mu!}\left(-\frac{1}{2F''(0)} \right)^\nu\frac{\dd^{2\nu}}{\dd t^{2\nu}}\left(G^\mu U\right)(0).  
\end{equation}
Then
\begin{equation}
\left|\int_I e^{i\lambda F(t)}U(t)-e^{i\lambda F(0)}\left( \frac{2\pi i}{\lambda F''(0)} \right)^{\frac{1}{2}}\sum_{j<k}\lambda^{-j}L_jU   \right|\leq C\lambda^{-k}||U||_{C^{2k}(I)}.
\end{equation}
\end{thm}

With the following definitions, this theorem can be applied to the integral in the $ x_2$ variable. We define
\begin{equation}
U=(-1)e^{i\sigma F_1}C :  V_{1;\tau}^+\otimes V_{2;\tau}^+\otimes\overline{V_{3;2\tau}^-},\quad F( x_2)=F_2( x_2),
\end{equation}
which then gives
\begin{equation}
G( x_2)=F_2( x_2)-F_2(0)-\frac{1}{2}F_2''(0) x_2^2=\frac{i}{2}\frac{\epsilon^3-3 x_1^2\epsilon}{( x_1^2+\epsilon^2)^3} x_2^4+\mathscr{O}( x_2^6).
\end{equation}
Here we treat $F_2$ as a function of $ x_2$.
With this $G$ we have
\begin{equation}
L_0(U)=U|_{ x_2=0},
\end{equation}
\begin{equation}
L_1(U)=\frac{1}{8\epsilon}\left[( x_1^2+\epsilon^2)\pr_2^2U+\frac{1}{8}\frac{3 x_1^2-\epsilon^2}{ x_1^2+\epsilon^2}U    \right]_{ x_2=0},
\end{equation}
and
\begin{equation}
L_2(U)=\frac{1}{128\epsilon^2}\left[( x_1^2+\epsilon^2)^2\pr_2^4U
+\frac{15(3 x_1^2-\epsilon^2)}{2}\pr_2^2U
+\frac{105(3 x_1^2-\epsilon^2)^2}{16( x_1^2+\epsilon^2)^2}U  \right]_{ x_2=0}.
\end{equation}
Theorem \ref{H-thm} gives  that
\begin{multline}\label{hormander}
\tau^{-2}\overline{\tau}^{-1}\frac{(-1)}{2}\int C : \nabla w_1\otimes\nabla w_2\otimes\nabla w_3\dd  x_2\\[5pt]
=\left(\pi \frac{ x_1^2+\epsilon^2}{2\lambda\epsilon} \right)^{\frac{1}{2}}e^{4i\sigma  x_0}e^{-4\sigma  x_1}\left[L_0(U)+\lambda^{-1}L_1(U) +\lambda^{-2}L_2(U) \right]\\[5pt]+\mathscr{O}(\lambda^{-\frac{7}{2}}).
\end{multline}

\subsection{$S$ is zero}\label{subsec-S-zero}

The first order (in $\lambda$) term in the expansion comes from 
\begin{multline}
U|_{ x_2=0}
=h_1(\vc x''')h_2(\vc x''')h_3(\vc x''')( x_1-i\epsilon)^{-\frac{1}{2}}( x_1^2+\epsilon^2)^{-\frac{1}{2}}\\[5pt]
\times C( x_0, x_1,0,\vc x''') : \vc\alpha\otimes\vc\alpha\otimes\vc\alpha+\mathscr{O}(\lambda^{-1}),
\end{multline}
so, recalling the decomposition $C=S+D$, we have
\begin{multline}
0=-\left(\frac{2\epsilon}{\pi}\right)^{\frac{1}{2}}\lim_{\lambda\to\infty}\lambda^{\frac{1}{2}}\tau^{-2}\overline{\tau}^{-1}\frac{1}{2}\int S:\nabla u_1\otimes\nabla u_2\otimes\nabla u_3\dd \vc x\\[5pt]
=\int e^{-4\sigma  x_1}\hat S(4\sigma, x_1,0,\vc x''') : \vc\alpha\otimes\vc\alpha\otimes\vc\alpha\\[5pt] \times( x_1-i\epsilon)^{-\frac{1}{2}}h_1(\vc x''')h_2(\vc x''')h_3(\vc x''')\dd  x_1\dd \vc x'''.
\end{multline}
In the above $\hat S$ is the Fourier transform of $S$ with respect to the $ x_0$ variable.

We need the following
\begin{lem}\label{jacobi}
Let $I$ be a bounded interval in $\R$ that does not contain the origin, $\mu>0$, and $\epsilon_0>0$. Let $f \in L^2(I)$ be such that 
\begin{equation}
\int_I f(t) (t-i\epsilon)^{-\mu}\dd t=0,\quad\forall \epsilon\in(0,\epsilon_0).
\end{equation}
Then $f= 0$.
\end{lem}

\begin{proof}
We have  $\forall \epsilon\in(0,\epsilon_0)$
\begin{multline}
0=\int_I f(t)(t-i\epsilon)^{-\mu}\dd t\\[5pt]
=\sum_{k=0}^\infty(i\epsilon)^k\frac{\mu(\mu+1)\cdots(\mu+k-1)}{k!}\int_I f(t)t^{-\mu} t^{-k}\dd t.
\end{multline}
It follows that for all $k=0,1,\ldots$
\begin{equation}
\int_I f(t)t^{-\mu} t^{-k}\dd t=0,
\end{equation}
so $f=0$, since $\mathrm{span}\,\{t^{-k}\}_{k=0}^\infty$ is dense in $L^2(I)$.
\end{proof}

Using Lemma \ref{jacobi} we can conclude that 
\begin{equation}
\int \hat S(4\sigma, x_1,0,\vc x''') : \vc\alpha\otimes\vc\alpha\otimes\vc\alpha h_1(\vc x''')h_2(\vc x''')h_3(\vc x''')\dd \vc x'''=0.
\end{equation}
It is  known that $\{h_1h_2h_3:h_j\text{ harmonic, } j=1,2,3\}$ is dense in $L^2(B_R^{\R^{n-2}})$, so we have
\begin{equation}
\hat S(4\sigma, x_1,0,0) : \vc\alpha\otimes\vc\alpha\otimes\vc\alpha=0,
\end{equation}
which implies that
\begin{equation}
S( x_0, x_1,0,0) : \vc\alpha\otimes\vc\alpha\otimes\vc\alpha=0.
\end{equation}
Since we can translate our coordinate system at will, it follows that
\begin{equation}
S(\vc x) : \vc\alpha\otimes\vc\alpha\otimes\vc\alpha=0, \quad\forall \vc x\in \mathcal{O}.
\end{equation}
Recall that $\vc\alpha=\vc e_0+i\vc e_1$. 
More generally, since the  coordinate system can also be arbitrarily rotated, we see that
\begin{equation}
S : \vc\alpha\otimes\vc\alpha\otimes\vc\alpha=0,
\end{equation}
for any $\vc\alpha\in\C^3$ such that $\vc\alpha\cdot\vc\alpha=0$.
By Hilbert's Nullstellensatz we then deduce that there exists $\vc c:\mathcal{O}\to\C^3$ such that
\begin{equation}\label{S-structure}
S:\vc\beta\otimes\vc\beta\otimes\vc\beta=(\vc c\cdot\vc\beta)(\vc\beta\cdot\vc\beta),\quad\forall\beta\in\C^3.
\end{equation}
Since $S$ is symmetric in all indices, by polarization, for vectors $\vc\beta_1,\vc\beta_2, \vc\beta_3\in\C^3$ we have
\begin{equation}\label{S-structure}
S:\vc\beta_1\otimes\vc\beta_2\otimes\vc\beta_3=\frac{1}{3}\left[(\vc c\cdot\vc\beta_1)(\vc\beta_2\cdot\vc\beta_3)+(\vc c\cdot\vc\beta_2)(\vc\beta_1\cdot\vc\beta_3)+(\vc c\cdot\vc\beta_3)(\vc\beta_1\cdot\vc\beta_2)   \right].
\end{equation}

We can isolate $S$ in \eqref{basic-integral-identity}. Since we have established \eqref{S-structure}, we can now use classical CGO solutions. To this end, let $\vc\xi,\vc\mu\in\R^3\setminus\{0\}$ such that $\vc\xi\perp\vc\mu$, $|\vc\mu|=1$,   and 
\begin{equation}
\vc\zeta_{\pm}=\frac{1}{2}\vc\xi\pm\frac{i}{2}|\vc\xi|\vc\mu.
\end{equation}
It is easy to check that
\begin{equation}
\vc\zeta_{+}\cdot\vc\zeta_{+}=\vc\zeta_{-}\cdot\vc\zeta_{-}=0,\quad \vc\zeta_{+}\cdot\vc\zeta_{-}=\frac{1}{2}|\vc\xi|^2.
\end{equation}
With these choices $e^{i\vc\zeta_{\pm}\cdot  \vc x}$ is a harmonic function.

Let 
\begin{equation}
u_1=u_2=e^{i\vc\zeta_{+}\cdot  \vc x},\quad u_3=e^{2i\vc\zeta_{-}\cdot  \vc x}.
\end{equation}
Note that by \eqref{basic-integral-identity} we have
\begin{multline}
3\int S:\nabla u_1\otimes\nabla u_2\otimes\nabla u_3\dd \vc x=
\int C:\big(\nabla u_1\otimes\nabla u_2\otimes\nabla u_3\\[5pt]+\nabla u_2\otimes\nabla u_3\otimes\nabla u_1+\nabla u_3\otimes\nabla u_2\otimes\nabla u_1  \big)\dd \vc x=0.
\end{multline}
Then
\begin{equation}
0=\hat S(\xi):\vc\zeta_{+}\otimes\vc\zeta_{+}\otimes\vc\zeta_{-}=\frac{1}{3}(\mathscr{F}{\vc c}\cdot\vc\zeta_{+})|\vc\xi|^2.
\end{equation}
Similarly, if we choose 
\begin{equation}
u_1=u_2=e^{i\vc\zeta_{-}\cdot  \vc x},\quad u_3=e^{2i\vc\zeta_{+}\cdot  \vc x},
\end{equation}
 we get
\begin{equation}
0=\mathscr{F} S(\xi):\vc\zeta_{-}\otimes\vc\zeta_{-}\otimes\vc\zeta_{+}=\frac{1}{3}(\mathscr{F}{\vc c}\cdot\vc\zeta_{-})|\vc\xi|^2.
\end{equation}
Adding and subtracting the two gives
\begin{equation}
\mathscr{F}{\vc c}\cdot\vc \xi=\mathscr{F}{\vc c}\cdot\vc\mu=0,
\end{equation}
so $\mathscr{F}{\vc c}=0$, which implies $S=0$. 

\subsection{$C$ is zero}\label{subsec-A-zero}

\subsubsection{The order $\lambda^{-\frac{3}{2}}$ term}

The expansions we have for the $V_{\tau}^\pm$ terms give an expansion for $C :  V_{1;\tau}^+\otimes V_{2;\tau}^+\otimes\overline{V_{3;2\tau}^-}$ in powers of $\lambda^{-1}$ and $ x_2$. The fact that $S=0$ implies that the $\lambda^0  x_2^0$ term is zero. In order to obtain the next order in $\lambda^{-1}$ information from \eqref{hormander}, we need to compute the $\lambda^{-1} x_2^0$ term of the expansion, as well as the $\lambda^{0} x_2^2$ term. A quick inspection of the expansions \eqref{V-plus} and \eqref{V-minus} shows that this term will be  a polynomial of order 1 in the arbitrary constants $q_1^j$. Since these can be chosen independently, the coefficient of each one will give us an independent identity. In order to simplify our computations, we will only write down terms that contain a  $q_1^j$ factor.

We begin with the order $\lambda^{-1} x_2^0$ of the expansion of $C :  V_{1;\tau}^+\otimes V_{2;\tau}^+\otimes\overline{V_{3;2\tau}^-}$. It is easy to see that this is
\begin{multline}
\chi^3 h_1h_2h_3C:\Big[
v_{0;1}^1|v_{0;0}|^2\vc e_2\otimes\vc\alpha\otimes\vc\alpha
+v_{0;1}^2|v_{0;0}|^2\vc\alpha\otimes\vc e_2\otimes\vc\alpha\\[5pt]
+\frac{1}{2}\overline{v_{0;1}^3}(v_{0;0})^2\vc\alpha\otimes\vc\alpha\otimes\vc e_2
\Big].
\end{multline}

The $\lambda^0  x_2^2$ term is made up of expressions obtained in two different ways, which we will refer to using the suggestive labels ``$ x_2^2\times 1\times 1$'' and ``$ x_2\times  x_2\times 1$''. The 
``$ x_2^2\times 1\times 1$'' contribution is
\begin{multline}
\chi^3 h_1h_2h_3C:\Big[
2i|v_{0;0}|^2\psi_2 v_{0;1}^1\vc e_2\otimes\vc\alpha\otimes\vc\alpha
+2i|v_{0;0}|^2\psi_2 v_{0;1}^2\vc\alpha\otimes\vc e_2\otimes\vc\alpha\\[5pt]
+2i(v_{0;0})^2\overline{\psi_2}\overline{v_{0;1}^3}\vc\alpha\otimes\vc\alpha\otimes\vc e_2
\Big].
\end{multline}
The ``$ x_2\times  x_2\times 1$'' contribution is
\begin{multline}
\chi^3 h_1h_2h_3C:\Big[
2i|v_{0;0}|^2\psi_2v_{0;1}^1\vc\alpha\otimes\vc e_2\otimes\vc\alpha
+2i|v_{0;0}|^2\overline{\psi}_2 v_{0;1}^1\vc\alpha\otimes\vc\alpha\otimes\vc e_2\\[5pt]
2i|v_{0;0}|^2\psi_2v_{0;1}^2\vc e_2\otimes\vc\alpha\otimes\vc\alpha
+2i|v_{0;0}|^2\overline{\psi}_2 v_{0;1}^2\vc\alpha\otimes\vc\alpha\otimes\vc e_2\\[5pt]
+2i(v_{0;0})^2\psi_2 \overline{v_{0;1}^3}\vc e_2\otimes\vc\alpha\otimes\vc\alpha
+2i(v_{0;0})^2\psi_2 \overline{v_{0;1}^3}\vc\alpha\otimes\vc e_2\otimes\vc\alpha
\Big].
\end{multline}

Before adding these terms up, observe that since $S=0$ we have
\begin{equation}
C:\left(\vc e_2\otimes\vc\alpha\otimes\vc\alpha+\vc\alpha\otimes\vc e_2\otimes\vc\alpha+\vc\alpha\otimes\vc\alpha\otimes\vc e_2\right)=0.
\end{equation}
The relevant part of the $\lambda^0  x_2^2$ term is then
\begin{multline}
2i(\overline{\psi}_2-\psi_2)\chi^3 h_1h_2h_3\Big[
|v_{0;0}|^2 v_{0;1}^1
+|v_{0;0}|^2 v_{0;1}^2
+(v_{0;0})^2 \overline{v_{0;1}^3}
\Big]\\[5pt]\times C:\vc\alpha\otimes\vc\alpha\otimes\vc e_2
\end{multline}

In order to compute the $\lambda^0$ order term in $L_1(U)$ we also need to account for two terms containing derivatives of $C$. The first is
\begin{multline}
\left.\pr_2^2(e^{i\sigma F_1}C) : V_{1;\tau}\otimes V_{2;\tau}\otimes\overline{V}_{3;-2\tau}\right|_{ x_2=0}\\[5pt]
=\left.h_1h_2h_3\pr_2^2(e^{i\sigma F_1}C):\vc\alpha\otimes\vc\alpha\otimes\vc\alpha\right|_{ x_2=0}+\mathscr{O}(\lambda^{-1})\\[5pt]=\mathscr{O}(\lambda^{-1}).
\end{multline}
This shows that it will contribute nothing to our computation. The one second is
\begin{multline}
\left.\pr_2(e^{i\sigma F_1}C) : \pr_2(V_{1;\tau}\otimes V_{2;\tau}\otimes\overline{V}_{3;-2\tau})\right|_{ x_2=0}
=\left.h_1h_2h_3\pr_2(e^{i\sigma F_1}C)\right|_{ x_2=0}\\[5pt]:\Big[
(v_{0;1}^1|v_{0;0}|^2+v_{0;1}^2|v_{0;0}|^2+\overline{v_{0;1}^3}(v_{0;0})^2)\vc\alpha\otimes\vc\alpha\otimes\vc\alpha\\[5pt]
+2i v_{0;0}|v_{0;0}|^2(\psi_2\vc e_2\otimes\vc\alpha\otimes\vc\alpha+
\psi_2\vc\alpha\otimes\vc e_2\otimes\vc\alpha+
\overline{\psi}_2\vc\alpha\otimes\vc\alpha\otimes\vc e_2)
\Big]\\[5pt]+\mathscr{O}(\lambda^{-1}),
\end{multline}
which also contributes nothing to the computation.

 Note  that
\begin{equation}\label{magic}
2i\frac{ x_1^2+\epsilon^2}{4\epsilon}(\overline{\psi_2}-\psi_2)=\frac{1}{2}.
\end{equation}
We therefore get that the $\lambda^{-1}$ order term of $L_0(U)+\lambda^{-1}L_1(U)$ is
\begin{multline}\label{first-order}
 h_1h_2h_3C|_{ x_2=0}:\Big[
|v_{0;0}|^2v_{0;1}^1\frac{1}{2}\left(\vc e_2\otimes\vc\alpha\otimes\vc\alpha-\vc\alpha\otimes\vc e_2\otimes\vc\alpha\right)\\[5pt]
+|v_{0;0}|^2v_{0;1}^2\frac{1}{2}\left(\vc\alpha\otimes\vc e_2\otimes\vc\alpha-\vc e_2\otimes\vc\alpha\otimes\vc\alpha\right)\\[5pt]
+(v_{0;0})^2\overline{v_{0;1}^3}\vc\alpha\otimes\vc\alpha\otimes\vc e_2
\Big]
\end{multline}

We first use the $q_1^3$ term to obtain the identity
\begin{multline}
0
=\int e^{-4\sigma  x_1}\hat C(4\sigma, x_1,0,\vc x''') : \vc\alpha\otimes\vc\alpha\otimes\vc e_2\\[5pt] \times ( x_1-i\epsilon)^{-\frac{1}{2}}( x_1+i\epsilon)^{-1}h_1(\vc x''')h_2(\vc x''')h_3(\vc x''')\dd  x_1\dd \vc x'''.
\end{multline}

In order to continue, we will need an analogue of Lemma \ref{jacobi}.
\begin{lem}\label{jacobi2}
Let $I$ be a bounded interval in $\R$ that does not contain the origin,  and let $\epsilon_0>0$. Let $f \in L^2(I)$ be such that 
\begin{equation}
\int_I f(t) (t-i\epsilon)^{-\frac{1}{2}}(t+i\epsilon)^{-1}\dd t=0,\quad\forall \epsilon\in(0,\epsilon_0).
\end{equation}
Then $f= 0$.
\end{lem}
\begin{proof}
The proof works in the same way as that of Lemma \ref{jacobi}. In order to have the conclusion, we only need to show that the Taylor expansion at $z=0$ of the function $(1-z)^{-\frac{1}{2}}(1+z)^{-1}$ does not contain any zero coefficients. To see that this is the case, first note that
\begin{equation}
(1-z)^{-\frac{1}{2}}=\sum_{j=0}^\infty\frac{(2j-1)!!}{2^jj!}z^j,\quad
(1+z)^{-1}=\sum_{j=0}^\infty (-1)^jz^j.
\end{equation}
Then
\begin{equation}
(1-z)^{-\frac{1}{2}}(1+z)^{-1}=
\sum_{j=0}^\infty (-1)^jz^j\sum_{l=0}^j(-1)^{l}\frac{(2l-1)!!}{2^ll!}.
\end{equation}
Take two successive terms from the second sum on the right hand side
\begin{equation}
\frac{(2l-1)!!}{2^ll!}-\frac{(2l+1)!!}{2^{l+1}(l+1)!}
=\frac{(2l-1)!!}{2^ll!}\left(1-\frac{2l+1}{2(l+1)}  \right)>0.
\end{equation}
We then see that $\sum_{l=0}^j(-1)^{l}\frac{(2l-1)!!}{2^ll!}$ can be written as a sum of positive terms, therefore it is non-zero for every $j$.
\end{proof}

We can now in the same way as above conclude that
\begin{equation}
C(\vc x):\vc\alpha\otimes\vc\alpha\otimes\vc e_2=0,\quad\forall \vc x\in\mathcal{O},
\end{equation}
which implies
\begin{equation}
2C(\vc x):\vc e_2\otimes\vc\alpha\otimes\vc\alpha=C(\vc x):(\vc e_2\otimes\vc\alpha\otimes\vc\alpha+\vc\alpha\otimes\vc e_2\otimes\vc\alpha)=0,\quad\forall \vc x\in\mathcal{O}.
\end{equation}
%
%
Putting everything together, we have
\begin{equation}\label{three-two}
C(\vc x):\vc e_2\otimes\vc\alpha\otimes\vc\alpha=C(\vc x):\vc\alpha\otimes\vc e_2\otimes\vc\alpha  = C(\vc x):\vc\alpha\otimes\vc\alpha\otimes\vc e_2= 0,\quad\forall \vc x\in\mathcal{O}.
\end{equation}

If we take the real part of \eqref{three-two} we get, for example, that
\begin{equation}\label{re1}
C(\vc x):\vc e_0\otimes\vc e_0\otimes\vc e_2=C(\vc x):\vc e_1\otimes\vc e_1\otimes\vc e_2,\quad\forall \vc x\in\mathcal{O}.
\end{equation}
We will later need to use the following consequence of \eqref{re1} that follows from a simple relabeling of axes
\begin{equation}\label{re2}
C(\vc x):\vc e_j\otimes\vc e_j\otimes\vc e_k=C(\vc x):\vc e_l\otimes\vc e_l\otimes\vc e_k,\quad\forall \vc x\in\mathcal{O}, j\neq k\neq l.
\end{equation} 

If we take the imaginary part of $C(\vc x):\vc\alpha\otimes\vc\alpha\otimes\vc e_2$ we get that
\begin{equation}\label{im1}
C(\vc x):\vc e_0\otimes\vc e_1\otimes\vc e_2=0,\quad\forall \vc x\in\mathcal{O}.
\end{equation}
Since the choice of coordinate axes is arbitrary, we can conclude that for any three mutually orthogonal vectors $\vc\mu_1, \vc\mu_2, \vc\mu_3\in\R^{1+n}$ we have
\begin{equation}\label{im2}
C(\vc x):\vc\mu_1\otimes\vc\mu_2\otimes\vc\mu_3=0,\quad\forall \vc x\in\mathcal{O}.
\end{equation}

\subsubsection{The $\lambda^{-\frac{5}{2}}$ term}

Here we will look at terms of order $\lambda^{-\frac{5}{2}}$ in \eqref{hormander}. As above, in order to avoid tedious computation we will only write down terms that contain a $q_1^jq_1^k$, $j\neq k$, term.

In terms of the expansion of $C :  V_{1;\tau}^+\otimes V_{2;\tau}^+\otimes\overline{V_{3;2\tau}^-}$, we need to compute the relevant terms of order $\lambda^{-2} x_2^0$, $\lambda^{-1} x_2^2$, and $\lambda^0  x_2^4$. All other orders do not contain terms that would satisfy our criteria. This also holds for all terms containing derivatives of $C$ that arise from $L_1(U)$ and $L_2(U)$.

There are two ways to obtain a $\lambda^{-2} x_2^0$ term, which we can denote ``$\lambda^{-2}\times 1
\times 1$'' and $\lambda^{-1}\times\lambda^{-1}\times 1$. Only the later contributes relevant terms. These are
\begin{multline}
\chi^3h_1h_2h_3 C: \Big[
v_{0;1}^1v_{0;1}^2\overline{v_{0;0}}\vc e_2\otimes\vc e_2\otimes\vc\alpha\\[5pt]
-\frac{1}{2}v_{0;1}^1\overline{v_{0;1}^3}v_{0;0}\vc e_2\otimes\vc\alpha\otimes\vc e_2
-\frac{1}{2}v_{0;1}^2\overline{v_{0;1}^3}v_{0;0}\vc\alpha\otimes\vc e_2\otimes\vc e_2
\Big]
\end{multline}

The relevant order $\lambda^{-1}$ term in $L_1(U)$ is obtained just from the $\pr_2^2$ part of that operator, the other term containing no $q_1^jq_1^k$ factors. It is therefore enough to identify the $\lambda^{-1} x_2^2$ terms in the expansion of $C :  V_{1;\tau}^+\otimes V_{2;\tau}^+\otimes\overline{V_{3;2\tau}^-}$. There are four ways to obtain $\lambda^{-1} x_2^2$. The first is picking  one term of order $\lambda^{-1} x_2^2$ from one of the $V_\tau^{\pm}$, multiplied by the zero order terms from the other two factors. This way produces no relevant terms. The other three ways could be abbreviated as ``$\tau^{-1}\times  x_2^2\times 1$'', ``$\tau^{-1} x_2\times  x_2\times 1$'', and ``$\tau^{-1}\times  x_2\times  x_2$''. 

The ``$\tau^{-1}\times  x_2^2\times 1$'' terms are
\begin{multline}
\chi^3 h_1h_2h_3 C:\Big[
2iv_{0;1}^1v_{0;1}^2\psi_2\overline{v_{0;0}}\vc e_2\otimes\vc e_2\otimes\vc\alpha
+2iv_{0;1}^1\overline{v_{0;1}^3}\overline{\psi}_2v_{0;0}\vc e_2\otimes\vc\alpha\otimes\vc e_2\\[5pt]
+2i v_{0;1}^2v_{0;1}^1\psi_2\overline{v_{0;0}}\vc e_2\otimes\vc e_2\otimes\vc\alpha
+2iv_{0;1}^2\overline{v_{0;1}^3}\overline{\psi}_2v_{0;0}\vc\alpha\otimes\vc e_2\otimes\vc e_2\\[5pt]
-i\overline{v_{0;1}^3}v_{0;1}^1\psi_2v_{0;0}\vc e_2\otimes\vc\alpha\otimes\vc e_2
-i\overline{v_{0;1}^3}v_{0;1}^2\psi_2v_{0;0}\vc\alpha\otimes\vc e_2\otimes\vc e_2
\Big].
\end{multline}

The ``$\tau^{-1} x_2\times  x_2\times 1$'' terms are
\begin{multline}
\chi^3 h_1h_2h_3 C:\Big[
\dot v_{0;1}^1v_{0;1}^2\overline{v_{0;0}}\vc e_1\otimes\vc\alpha\otimes\vc\alpha
+\dot v_{0;1}^1\overline{v_{0;1}^3} v_{0;0}\vc e_1\otimes\vc\alpha\otimes\vc\alpha\\[5pt]
+\dot v_{0;1}^2v_{0;1}^1\overline{v_{0;0}}\vc\alpha\otimes\vc e_1\otimes\vc\alpha
+\dot v_{0;1}^2\overline{v_{0;1}^3}v_{0;0}\vc\alpha\otimes\vc e_1\otimes\vc\alpha\\[5pt]
-\frac{1}{2}\overline{\dot v_{0;1}^3}v_{0;1}^1 v_{0;0}\vc\alpha\otimes\vc\alpha\otimes\vc e_1
-\frac{1}{2}\overline{\dot v_{0;1}^3}v_{0;1}^2 v_{0;0}\vc\alpha\otimes\vc\alpha\otimes\vc e_1
\Big].
\end{multline}

The ``$\tau^{-1}\times  x_2\times  x_2$'' terms are
\begin{multline}
\chi^3 h_1h_2h_3 C:\Big[
\dot v_{0;0}v_{0;1}^2\overline{v_{0;1}^3}\vc e_1\otimes\vc\alpha\otimes\vc\alpha
+\dot v_{0;0}v_{0;1}^1\overline{v_{0;1}^3}\vc\alpha\otimes\vc e_1\otimes\vc\alpha
-\frac{1}{2}\overline{\dot v_{0;0}}v_{0;1}^1v_{0;1}^2\vc\alpha\otimes\vc\alpha\otimes\vc e_1\\[5pt]
+2i v_{0;1}^1v_{0;1}^2\overline{v_{0;0}}\overline{\psi}_2 \vc e_2\otimes\vc\alpha\otimes\vc e_2
+2i v_{0;1}^1\overline{v_{0;1}^3} v_{0;0}\psi_2 \vc e_2\otimes\vc e_2\otimes\vc\alpha\\[5pt]
+2i v_{0;1}^2v_{0;1}^1\overline{v_{0;0}}\overline{\psi}_2 \vc\alpha\otimes\vc e_2\otimes\vc e_2
+2i v_{0;1}^2\overline{v_{0;1}^3} v_{0;0}\psi_2 \vc e_2\otimes\vc e_2\otimes\vc\alpha\\[5pt]
-i \overline{v_{0;1}^3} v_{0;1}^1v_{0;0}\psi_2 \vc\alpha\otimes\vc e_2\otimes\vc e_2
-i \overline{v_{0;1}^3} v_{0;1}^2v_{0;0}\psi_2 \vc e_2\otimes\vc\alpha\otimes\vc e_2
\Big].
\end{multline}

It may at first seem that  a few terms containing factors like $C:\nabla'''h_1\otimes\vc \alpha\otimes\vc\alpha$ or $C:\nabla'''h_1\otimes \vc e_2\otimes\vc\alpha$ should also appear here. Note however that by \eqref{three-two}, the fact that $\nabla'''h_1\perp\vc\alpha$,  and the freedom to choose coordinate axes we can conclude that factors like $C:\nabla'''h_1\otimes\vc \alpha\otimes\vc\alpha$ are in fact zero. A similar comment can be made about the other kind of terms. These  can be seen to be zero by  \eqref{im2}, since $\nabla'''h_j$, $\vc e_0$, $\vc e_1$, $\vc e_2$ are mutually orthogonal.

The relevant order $\lambda^0$ term in $L_2(U)$ is obtained from the $\pr_2^4$ part of the operator acting on the $\lambda^0 x_2^4$ term in the expansion of $C :  V_{1;\tau}^+\otimes V_{2;\tau}^+\otimes\overline{V_{3;2\tau}^-}$. There are   ways of obtaining $\lambda^0 x_2^4$, namely ``$ x_2^4\times1\times1$'', ``$ x_2^3\times  x_2\times1$'', ``$ x_2^2\times  x_2^2\times 1$'', and ``$ x_2^2\times  x_2\times  x_2$''. There are no relevant terms that can be obtained via  ``$ x_2^4\times1\times1$''.

The ``$ x_2^3\times  x_2\times1$'' terms are
\begin{multline}
\chi^3h_1h_2h_3C:\Big[
i\dot\psi_2 \overline{v_{0;0}}v_{0;1}^1v_{0;1}^2\vc e_1\otimes\vc\alpha\otimes\vc\alpha
+i\dot\psi_2 v_{0;0}v_{0;1}^1\overline{v_{0;1}^3}\vc e_1\otimes\vc\alpha\otimes\vc\alpha\\[5pt]
+i\dot\psi_2 \overline{v_{0;0}}v_{0;1}^2v_{0;1}^1\vc\alpha\otimes\vc e_1\otimes\vc\alpha
+i\dot\psi_2 v_{0;0}v_{0;1}^2\overline{v_{0;1}^3}\vc\alpha\otimes\vc e_1\otimes\vc\alpha\\[5pt]
+i\overline{\dot\psi}_2v_{0;0}\overline{v_{0;1}^3}v_{0;1}^1\vc\alpha\otimes\vc\alpha\otimes\vc e_1
+i\overline{\dot\psi}_2v_{0;0}\overline{v_{0;1}^3}v_{0;1}^2\vc\alpha\otimes\vc\alpha\otimes\vc e_1
\Big].
\end{multline}

The ``$ x_2^2\times  x_2^2\times 1$'' terms are
\begin{multline}
(-4)\chi^3h_1h_2h_3C:\Big[
(\psi_2)^2\overline{v_{0;0}}v_{0;1}^1v_{0;1}^2\vc e_2\otimes\vc e_2\otimes\vc\alpha
+|\psi_2|^2{v_{0;0}}v_{0;1}^1\overline{v_{0;1}^3}\vc e_2\otimes\vc\alpha\otimes\vc e_2\\[5pt]
+|\psi_2|^2{v_{0;0}}v_{0;1}^2\overline{v_{0;1}^3}\vc\alpha\otimes\vc e_2\otimes\vc e_2
\Big].
\end{multline}

The ``$ x_2^2\times  x_2\times  x_2$'' terms are
\begin{multline}
\chi^3h_1h_2h_3C:\Big[
iv_{0;0}\dot\psi_2v_{0;1}^2\overline{v_{0;1}^3}\vc e_1\otimes\vc\alpha\otimes\vc\alpha
-4\overline{v_{0;0}}|\psi_2|^2v_{0;1}^1v_{0;1}^2\vc e_2\otimes\vc\alpha\otimes\vc e_2\\[5pt]
-4{v_{0;0}}(\psi_2)^2v_{0;1}^1\overline{v_{0;1}^3}\vc e_2\otimes\vc e_2\otimes\vc\alpha\\[5pt]
+iv_{0;0}\dot\psi_2v_{0;1}^1\overline{v_{0;1}^3}\vc\alpha\otimes\vc e_1\otimes\vc\alpha
-4\overline{v_{0;0}}|\psi_2|^2v_{0;1}^2v_{0;1}^1\vc\alpha\otimes\vc e_2\otimes\vc e_2\\[5pt]
-4{v_{0;0}}(\psi_2)^2v_{0;1}^2\overline{v_{0;1}^3}\vc e_2\otimes\vc e_2\otimes\vc\alpha\\[5pt]
+i\overline{v_{0;0}}\overline{\dot\psi}_2v_{0;1}^1{v_{0;1}^2}\vc\alpha\otimes\vc\alpha\otimes\vc e_1
-4{v_{0;0}}|\psi_2|^2\overline{v_{0;1}^3}v_{0;1}^1\vc\alpha\otimes\vc e_2\otimes\vc e_2\\[5pt]
-4{v_{0;0}}|\psi_2|^2\overline{v_{0;1}^3}v_{0;1}^2\vc e_2\otimes\vc\alpha\otimes\vc e_2
\Big]
\end{multline}

We must now put together all the above contributions to the order $\lambda^{-2}$ term of $L_0(U)+\lambda^{-1}L_1(U)+\lambda^{-2}L_2(U)$. We choose parameters $q_1^1=q_1^2=q$. Furthermore, we will from now on only consider terms that contain the factor $qq_1^3$. 

Using the fact that $S=0$, which implies the symmetries 
\begin{equation}
2C:\vc e_1\otimes\vc\alpha\otimes\vc\alpha=2C:\vc\alpha\otimes\vc e_1\otimes\vc\alpha=-C:\vc\alpha\otimes\vc\alpha\otimes\vc e_1,
\end{equation}
\begin{equation}
2C:\vc e_2\otimes\vc\alpha\otimes\vc e_2=2C:\vc\alpha\otimes\vc e_2\otimes\vc e_2=-C:\vc e_2\otimes\vc e_2\otimes\vc\alpha,
\end{equation}
we compute the contribution from $L_2(U)$ to the coefficient of $\lambda^{-2}qq_1^3$ to be
\begin{multline}
(-1)\frac{3}{4\epsilon}( x_1-i\epsilon)^{-\frac{1}{2}}( x_1^2+\epsilon^2)^{-\frac{3}{2}}h_1h_2h_3C|_{ x_2=0}\\[5pt]:\Big[ x_1\vc\alpha\otimes\vc\alpha\otimes\vc e_1+( x_1+i\epsilon)\vc e_2\otimes\vc e_2\otimes\vc\alpha
\Big].
\end{multline}  
Similarly we compute the contribution from $L_1(U)$ to the coefficient of $\lambda^{-2}qq_1^3$ to be
\begin{multline}
\frac{1}{4\epsilon}( x_1-i\epsilon)^{-\frac{1}{2}}( x_1^2+\epsilon^2)^{-\frac{3}{2}}h_1h_2h_3C|_{ x_2=0}\\[5pt]
:\Big[\frac{1}{2}\big(4( x_1+i\epsilon)+3( x_1-i\epsilon)  \big)\vc\alpha\otimes\vc\alpha\otimes\vc e_1\\[5pt]+i\big(3( x_1+i\epsilon)-( x_1-i\epsilon)\big)\vc e_2\otimes\vc e_2\otimes\vc\alpha
\Big].
\end{multline}  
The contribution from $L_0(U)$ to the coefficient of $\lambda^{-2}qq_1^3$ is
\begin{equation}
\frac{1}{2}( x_1-i\epsilon)^{-\frac{1}{2}}( x_1^2+\epsilon^2)^{-\frac{3}{2}}h_1h_2h_3C|_{ x_2=0}:\vc e_2\otimes\vc e_2\otimes\vc\alpha.
\end{equation}
Adding these we obtain that the coefficient of $\lambda^{-2}qq_1^3$  in $L_0(U)+\lambda^{-1}L_1(U)+\lambda^{-2}L_2(U)$ is
\begin{multline}
\frac{1}{4\epsilon}( x_1+i\epsilon)( x_1-i\epsilon)^{-\frac{1}{2}}( x_1^2+\epsilon^2)^{-\frac{3}{2}}h_1h_2h_3C|_{ x_2=0}\\[5pt]
:\Big[
\frac{1}{2}\vc\alpha\otimes\vc\alpha\otimes\vc e_1
+(3-2i)\vc e_2\otimes\vc e_2\otimes\vc\alpha
\Big]
\end{multline}

From \eqref{hormander} we now get the identity
\begin{multline}
0
=\int e^{-4\sigma  x_1}\hat C(4\sigma, x_1,0,\vc x''') : \Big[
\frac{1}{2}\vc\alpha\otimes\vc\alpha\otimes\vc e_1
+(3-2i)\vc e_2\otimes\vc e_2\otimes\vc\alpha
\Big]\\[5pt] \times ( x_1-i\epsilon)^{-\frac{3}{2}}
h_1(\vc x''')h_2(\vc x''')h_3(\vc x''')\dd  x_1\dd \vc x'''.
\end{multline}

Repeating the same method of proof we have used above, we can conclude that
\begin{equation}\label{five-two}
C(\vc x):\Big[
\frac{1}{2}\vc\alpha\otimes\vc\alpha\otimes\vc e_1
+(3-2i)\vc e_2\otimes\vc e_2\otimes\vc\alpha
\Big]=0,\quad\forall \vc x\in\mathcal{O}.
\end{equation}
If we take the real part of this identity we get
\begin{equation}
C(\vc x):\Big[
\frac{1}{2}\vc e_0\otimes\vc e_0\otimes\vc e_1+3\vc e_2\otimes\vc e_2\otimes\vc e_0
+2\vc e_2\otimes\vc e_2\otimes\vc e_1
\Big]=0,\quad \forall \vc x\in\mathcal{O}.
\end{equation}
Using \eqref{re2} this becomes
\begin{equation}\label{re3}
C(\vc x):\Big[
3\vc e_2\otimes\vc e_2\otimes\vc e_0
+\frac{5}{2}\vc e_2\otimes\vc e_2\otimes\vc e_1
\Big]=0,\quad \forall \vc x\in\mathcal{O}.
\end{equation}

If we take the imaginary part of \eqref{five-two} we get
\begin{equation}
C(\vc x):\Big[
\vc e_0\otimes\vc e_1\otimes\vc e_1+3\vc e_2\otimes\vc e_2\otimes\vc e_1
-2\vc e_2\otimes\vc e_2\otimes\vc e_0
\Big]=0,\quad\forall \vc x\in\mathcal{O}.
\end{equation}
Using \eqref{re2} again we have
\begin{equation}
C(\vc x):\Big[
\vc e_0\otimes\vc e_2\otimes\vc e_2+3\vc e_2\otimes\vc e_2\otimes\vc e_1
-2\vc e_2\otimes\vc e_2\otimes\vc e_0
\Big]=0,\quad\forall \vc x\in\mathcal{O}.
\end{equation}
Since $S=0$
\begin{equation}
C:\vc e_0\otimes\vc e_2\otimes\vc e_2=-\frac{1}{2}C:\vc e_2\otimes\vc e_2\otimes\vc e_0,
\end{equation}
so
\begin{equation}\label{im3}
C(\vc x):\Big[
3\vc e_2\otimes\vc e_2\otimes\vc e_0-\frac{5}{2}\vc e_2\otimes\vc e_2\otimes\vc e_1
\Big]=0,\quad \forall \vc x\in\mathcal{O}.
\end{equation}
From \eqref{re3} and \eqref{im3} we conclude that
\begin{equation}\label{re-im}
C(\vc x):\vc e_2\otimes\vc e_2\otimes\vc e_0=0,\quad\forall \vc x\in\mathcal{O}.
\end{equation}

Combining the fact that $S=0$ with \eqref{im2}, \eqref{re-im}, and using the freedom we have to relabel coordinate axes, we have that for all $\vc x\in\mathcal{O}$, $j,k,l=0,\ldots, n$
\begin{equation}
C(\vc x):\vc e_j\otimes\vc e_j\otimes\vc e_j=0,
\end{equation}
\begin{equation}
C(\vc x):\vc e_j\otimes\vc e_k\otimes\vc e_k=C(\vc x):\vc e_k\otimes\vc e_j\otimes\vc e_k=C(\vc x):\vc e_k\otimes\vc e_k\otimes\vc e_j=0, \, j\neq k,
\end{equation}
\begin{equation}
C(\vc x):\vc e_j\otimes\vc e_k\otimes\vc e_l=0,\quad j\neq k\neq l\neq j.
\end{equation}
This amounts to $C=0$.

\paragraph{Acknowledgments:} C\u at\u alin I. C\^arstea was supported by NSF of China under grant 11931011. Ali Feizmohammadi was supported by EPSRC grant EP/P01593X/1.

\newpage
\appendix

\section{Summary of the results of computations}\label{appendix-results}

We list here for easy reference the results of some of our computations.

\subsection{Expansion of solutions}

Below we give the first few terms in the expansions of $\Psi( x_1, x_2)$ and $a_\tau( x_1, x_2)$, which both appear in the construction of the harmonic functions $u_\tau^\pm$.

\begin{equation}
u_\tau^\pm(\vc x)=e^{\pm\lambda  x_0}\left( e^{\pm i\sigma  x_0} e^{i\tau\Psi( x_1, x_2)}a_\tau( x_1, x_2)+r_\tau(\vc x)   \right)
\end{equation}
\begin{equation}
\Psi( x_1, x_2)=\sum_{j=0}^M\psi_j( x_1) x_2^j
\end{equation}
\begin{equation}
a_\tau( x_1, x_2)=
\chi(\frac{ x_2}{\delta})h(\vc x''')\sum_{k=0}^Mv_k( x_1, x_2)\tau^{-k}
\end{equation}
\begin{equation}
v_k( x_1, x_2)=\sum_{j=0}^M v_{k;j}( x_1) x_2^j
\end{equation}

\begin{equation}
\psi_0( x_1)= x_1,\quad \psi_1( x_1)=0
\end{equation}
\begin{equation}
\psi_2( x_1)=\frac{1}{2}( x_1-i\epsilon)^{-1}
\end{equation}
\begin{equation}
\psi_3=p_3( x_1-i\epsilon)^{-3},\quad p_3\in\C
\end{equation}
\begin{equation}
\psi_4( x_1)=-\frac{1}{8}( x_1-i\epsilon)^{-3}+\frac{9}{2}p_3^2( x_1-i\epsilon)^{-5}+p_4( x_1-i\epsilon)^{-4},\quad p_4\in\C.
\end{equation}
\begin{multline}
\psi_5( x_1)
=27p_3^3( x_1-i\epsilon)^{-7}+12p_3p_4( x_1-i\epsilon)^{-6}
\\[5pt]
 +p_5( x_1-i\epsilon)^{-5},\quad p_5\in\C.
\end{multline}

\begin{equation}
v_{0;0}=( x_1-i\epsilon)^{-\frac{1}{2}}.
\end{equation}
\begin{equation}
v_{0;1}=3p_3( x_1-i\epsilon)^{-\frac{5}{2}}+q_1( x_1-i\epsilon)^{-\frac{3}{2}},\quad q_1\in\C.
\end{equation}


\subsection{The $V_\tau^\pm$ terms}

Here we record the expansions of the quantities $V_\tau^\pm$.

\begin{multline}
V_\tau^+=\tau^{-1} e^{-\tau  x_0}e^{-i\tau\Psi}\nabla u_\tau^+\\[5pt]
=\chi h\Big[v_{0;0}\vc\alpha+ x_2\left[v_{0;1}\vc\alpha+2iv_{0;0}\psi_2\vc e_2  \right]\\[5pt]
+ x_2^2\left[v_{0;2}\vc\alpha+iv_{0;0}\dot\psi_2\vc e_1+i(2v_{0;1}\psi_2+3v_{0;0}\psi_3)\vc e_2  \right]\\[5pt]
+\tau^{-1}\left[v_{1;0}\vc\alpha +\dot v_{0;0}\vc e_1+v_{0;1}\vc e_2 \right]\\[5pt]
+ x_2^3\left[v_{0;3}\vc\alpha+i(v_{0;1}\dot\psi_2+v_{0;0}\dot\psi_3)\vc e_1 +i(2v_{0;2}\psi_2+3v_{0;1}\psi_3+4v_{0;0}\psi_4)\vc e_2   \right]\\[5pt]
+\tau^{-1} x_2\left[ v_{1;1}\vc\alpha+\dot v_{0;1}\vc e_1 +(2iv_{1;0}\psi_2+2v_{0;2})\vc e_2 \right]\Big]\\[5pt]
+\tau^{-1}\left[ \delta^{-1}\dot\chi h\vc e_2+\chi\nabla''' h\right](v_{0;0}+v_{0;1} x_2)+\cdots.
\end{multline}
\begin{multline}
(-1)\overline{V_\tau^-}:=(-1)\overline{\tau^{-1} e^{\tau  x_0}e^{i\tau\Psi}\nabla u_\tau^-}\\[5pt]
=\chi h\Big[\overline{v_{0;0}}\vc\alpha+ x_2\left[\overline{v_{0;1}}\vc\alpha+2i\overline{v_{0;0}}\overline{\psi_2}\vc e_2  \right]\\[5pt]
+ x_2^2\left[\overline{v_{0;2}}\vc\alpha+i\overline{v_{0;0}}\overline{\dot\psi_2}\vc e_1+i(2\overline{v_{0;1}}\overline{\psi_2}+3\overline{v_{0;0}}\overline{\psi_3})\vc e_2  \right]\\[5pt]
+\overline{\tau^{-1}}\left[\overline{v_{1;0}}\vc\alpha -\overline{\dot v_{0;0}}\vc e_1-\overline{v_{0;1}}\vc e_2 \right]\\[5pt]
+ x_2^3\left[\overline{v_{0;3}}\vc\alpha+i(\overline{v_{0;1}}\overline{\dot\psi_2}+\overline{v_{0;0}}\overline{\dot\psi_3})\vc e_1 +i(2\overline{v_{0;2}}\overline{\psi_2}+3\overline{v_{0;1}}\overline{\psi_3}+4\overline{v_{0;0}}\overline{\psi_4})\vc e_2   \right]\\[5pt]
+\overline{\tau^{-1}} x_2\left[ \overline{v_{1;1}}\vc\alpha-\overline{\dot v_{0;1}}\vc e_1 +(2i\overline{v_{1;0}}\overline{\psi_2}-2\overline{v_{0;2}})\vc e_2 \right]\Big]\\[5pt]
-\overline{\tau^{-1}}\left[ \delta^{-1}\dot\chi h\vc e_2+\chi\nabla''' h\right](\overline{v_{0;0}}+\overline{v_{0;1}} x_2)+\cdots.
\end{multline}

\bibliography{quasilinear}

\begin{thebibliography}{10}

\bibitem{BU}
V.~M. Babich and V.V. Ulin.
\newblock Complex space-time ray method and “quasifotons”.
\newblock {\em Zapiski Nauchnykh Seminarov POMI}, 117:5--12, 1981.

\bibitem{Ca}
A.~P. Calder{\'o}n.
\newblock On an inverse boundary value problem.
\newblock In {\em Seminar on {N}umerical {A}nalysis and its {A}pplications to
  {C}ontinuum {P}hysics ({R}io de {J}aneiro, 1980)}, pages 65--73. Soc. Brasil.
  Mat., Rio de Janeiro, 1980.

\bibitem{C}
C.~I. C{\^a}rstea.
\newblock On an inverse boundary value problem for a nonlinear time harmonic
  {M}axwell system.
\newblock {\em arXiv preprint arXiv:1804.09586}, 2018.

\bibitem{CK}
C.~I. C{\^a}rstea and M.~Kar.
\newblock Recovery of coefficients for a weighted p-{L}aplacian perturbed by a
  linear second order term.
\newblock {\em arXiv preprint arXiv:2001.01436}, 2020.

\bibitem{CNV}
C.~I. C{\^a}rstea, G.~Nakamura, and M.~Vashisth.
\newblock Reconstruction for the coefficients of a quasilinear elliptic partial
  differential equation.
\newblock {\em Applied Mathematics Letters}, 2019.

\bibitem{DKLS}
D.~Dos Santos~Ferreira, Y.~Kurylev, M.~Lassas, and M.~Salo.
\newblock The {C}alder{\'o}n problem in transversally anisotropic geometries.
\newblock {\em Journal of the European Mathematical Society},
  18(11):2579--2626, 2016.

\bibitem{EPS}
H.~Egger, J.-F. Pietschmann, and M.~Schlottbom.
\newblock Simultaneous identification of diffusion and absorption coefficients
  in a quasilinear elliptic problem.
\newblock {\em Inverse Problems}, 30(3):035009, 2014.

\bibitem{FO}
A.~Feizmohammadi and L.~Oksanen.
\newblock An inverse problem for a semi-linear elliptic equation in riemannian
  geometries.
\newblock {\em Journal of Differential Equations}, 269(6):4683--4719, 2020.

\bibitem{GU}
A.~Greenleaf and G.~Uhlmann.
\newblock Local uniqueness for the {D}irichlet-to-{N}eumann map via the
  two-plane transform.
\newblock {\em Duke Mathematical Journal}, 108(3):599--617, 2001.

\bibitem{HS}
D.~Hervas and Z.~Sun.
\newblock An inverse boundary value problem for quasilinear elliptic equations.
\newblock {\em Communications in Partial Differential Equations},
  27(11-12):2449--2490, 2002.

\bibitem{H1}
L.~H{\"o}rmander.
\newblock {\em The Analysis of Linear Partial Differential Operators. I,
  Distribution Theory and Fourier Analysis}.
\newblock Grundlehren Der Mathematischen Wissenschaften. Springer, 2 edition,
  1990.

\bibitem{I1}
V.~Isakov.
\newblock On uniqueness in inverse problems for semilinear parabolic equations.
\newblock {\em Archive for Rational Mechanics and Analysis}, 124(1):1--12,
  1993.

\bibitem{I2}
V.~Isakov.
\newblock Uniqueness of recovery of some quasilinear partial differential
  equations.
\newblock {\em Communications in Partial Differential Equations},
  26(11-12):1947--1973, 2001.

\bibitem{IN}
V.~Isakov and A.~I. Nachman.
\newblock Global uniqueness for a two-dimensional semilinear elliptic inverse
  problem.
\newblock {\em Transactions of the American Mathematical Society},
  347(9):3375--3390, 1995.

\bibitem{IS}
V.~Isakov and J.~Sylvester.
\newblock Global uniqueness for a semilinear elliptic inverse problem.
\newblock {\em Communications on Pure and Applied Mathematics},
  47(10):1403--1410, 1994.

\bibitem{KKL}
A.~Kachalov, Y.~Kurylev, and M.~Lassas.
\newblock {\em Inverse boundary spectral problems}.
\newblock CRC Press, 2001.

\bibitem{KN}
H.~Kang and G.~Nakamura.
\newblock Identification of nonlinearity in a conductivity equation via the
  {D}irichlet-to-{N}eumann map.
\newblock {\em Inverse Problems}, 18(4):1079, 2002.

\bibitem{KSa}
C.~Kenig and M.~Salo.
\newblock The calder{\'o}n problem with partial data on manifolds and
  applications.
\newblock {\em Analysis \& PDE}, 6(8):2003--2048, 2014.

\bibitem{KV}
{R. V.} Kohn and M.~Vogelius.
\newblock Identification of an unknown conductivity by means of measurements at
  the boundary.
\newblock In {David W.} McLaughlin, editor, {\em SIAM - AMS Proceedings}, SIAM
  - AMS Proceedings, pages 113--123. American Mathematical Soc, December 1984.

\bibitem{KU}
K.~Krupchyk and G.~Uhlmann.
\newblock Partial data inverse problems for semilinear elliptic equations with
  gradient nonlinearities.
\newblock {\em arXiv preprint arXiv:1909.08122}, 2019.

\bibitem{LLLS2}
M.~Lassas, T.~Liimatainen, Y.-H. Lin, and M.~Salo.
\newblock Partial data inverse problems and simultaneous recovery of boundary
  and coefficients for semilinear elliptic equations.
\newblock {\em arXiv preprint arXiv:1905.02764}, 1905.

\bibitem{LLLS1}
M.~Lassas, T.~Liimatainen, Y.-H. Lin, and M.~Salo.
\newblock Inverse problems for elliptic equations with power type
  nonlinearities.
\newblock {\em arXiv preprint arXiv:1903.12562}, 2019.

\bibitem{MU}
C.~Munoz and G.~Uhlmann.
\newblock The {C}alder\'{o}n problem for quasilinear elliptic equations.
\newblock {\em Annales de l'Institut Henri Poincar{\'e} C, Analyse non
  lin{\'e}aire}, 2020.

\bibitem{Ralston}
J.~Ralston.
\newblock Gaussian beams and the propagation of singularities.
\newblock {\em Studies in partial differential equations}, 23(206):C248, 1982.

\bibitem{Sh}
R.~Shankar.
\newblock Recovering a quasilinear conductivity from boundary measurements.
\newblock {\em arXiv preprint arXiv:1910.07890}, 2019.

\bibitem{S1}
Z.~Sun.
\newblock On a quasilinear inverse boundary value problem.
\newblock {\em Mathematische Zeitschrift}, 221(1):293--305, 1996.

\bibitem{S3}
Z.~Sun.
\newblock Anisotropic inverse problems for quasilinear elliptic equations.
\newblock In {\em Journal of Physics: Conference Series}, volume~12, page 015.
  IOP Publishing, 2005.

\bibitem{S2}
Z.~Sun.
\newblock An inverse boundary-value problem for semilinear elliptic equations.
\newblock {\em Electronic Journal of Differential Equations (EJDE)[electronic
  only]}, 2010:Paper--No, 2010.

\bibitem{SuU}
Z.~Sun and G.~Uhlmann.
\newblock Inverse problems in quasilinear anisotropic media.
\newblock {\em American Journal of Mathematics}, 119(4):771--797, 1997.

\end{thebibliography}
\bibliographystyle{plain}

\end{document}